\RequirePackage{fix-cm}
\documentclass[smallextended,natbib]{svjour3}

\smartqed  % flush right qed marks, e.g. at end of proof
\usepackage{mathptmx}      % use Times fonts if available on your TeX system

\usepackage[english]{babel}
\usepackage[top=2.5cm, left=2.25cm, right=2.25cm, bottom=2cm]{geometry}
\usepackage{amsmath,amsfonts,amssymb,dsfont}
\usepackage{graphicx}
\usepackage{pgfplots}
\usepackage{epstopdf}
\usepackage[hyperfootnotes=false]{hyperref}

\newcommand{\B}{\mathbb{B}}
\newcommand{\E}{\mathbb{E}}
\newcommand{\N}{\mathbb{N}}
\renewcommand{\P}{\mathbb{P}}
\newcommand{\R}{\mathbb{R}}
\newcommand{\Po}{\text{Po}}
\newcommand{\diam}{\text{diam}}
\newcommand{\diag}{\text{diag}}
\newcommand{\PRM}{\text{PRM}}
\newcommand{\ind}{\mathds{1}}
\newcommand{\con}{\eta}

\newcommand{\boundSet}{B}
\newcommand{\ZVhelp}{V}

\spnewtheorem{condition}{Condition}{\bf}{\it}

\addto\extrasenglish{

}

 \journalname{Extremes}
\begin{document}

\title{Limit laws for the diameter of a set of random points from a distribution supported by a smoothly bounded set
}



\author{Michael Schrempp %etc.
}


\institute{M. Schrempp \at
              Karlsruhe Institute of Technology,
              Institute of Stochastics,
              Englerstr. 2,
              76131 Karlsruhe, Germany\\
              Tel.: +49-721-60843264\\
              \email{schrempp@kit.edu}           %  \\
}


%
\maketitle

\vspace{-2cm}
\begin{abstract}
We study the asymptotic behavior of the maximum interpoint distance of random points in a $d$-dimensional set with a unique diameter and a smooth boundary at the poles.
Instead of investigating only a fixed number of $n$ points as $n$ tends to infinity, we consider the much more general setting in which the random points are the supports of appropriately defined Poisson processes.
The main result covers the case of uniformly distributed points within a $d$-dimensional ellipsoid with a unique major axis. Moreover, several generalizations of the main result are established, for example a limit law for the maximum interpoint distance of random points from a Pearson type II distribution.

\end{abstract}

\section{Introduction}
\label{sec_introduction}
For some fixed integer $d \ge 2$, let $Z,Z_1,Z_2,\ldots$ be independent and identically distributed (i.i.d.) $d$-dimensional random vectors, defined  on a common probability space $(\Omega,{\cal A},\P)$. We assume that the distribution $\P_Z$ of $Z$ is absolutely continuous with respect to Lebesgue measure.
Writing $|\cdot |$ for the Euclidean norm on $\R^d$, the asymptotical behavior of the so-called maximum interpoint distance
$$
M_n := \max\limits_{1 \le i , j \le n}|Z_i - Z_j|
$$
as $n $ tends to infinity
has been a topic of interest for more than 20 years. This behavior is closely related to the support $S\subset \R^d$ of $\P_{Z}$, which is the smallest closed set $C$ satisfying $\P_Z(C)=1$. Writing
$$
\diam(K):=\sup_{x,y \in K}|x-y|
$$
for the diameter of a set $K \subset \R^d$, we obviously have
$
M_n \overset{\text{a.s.}}{\longrightarrow } \diam(S)
$
as $n \to \infty $, but finding sequences $(a_n)_{n \in \N} $ and $(b_n)_{n \in \N}$ so that $a_n(b_n-M_n)$ has a non-degenerate limit distribution as $n \to \infty $ is a much more difficult problem, which has hitherto been solved only in a few special cases.
We deliberately discard the case $d=1$ in what follows since then
\begin{equation*}
M_n = \max_{1\le i \le n}Z_i - \min_{1\le i \le n}Z_i
\end{equation*}
is the well-studied sample range.
Results obtained so far mostly cover the case that $\P_{Z}$ is spherically symmetric, and they may roughly be classified according to whether $\P_{Z}$ has an unbounded or a bounded support. If $Z$ has a spherically symmetric normal distribution, \cite{Matthews1993} obtained a Gumbel limit distribution
for $M_n$, and \cite{Henze1996} generalized this result to the case that $Z$ has a spherically symmetric Kotz type distribution.
An even more general spherically symmetric setting with a Gumbel limit distribution has been studied by \cite{RaoJanson2015}.
\cite{HenzeLao2010} studied unbounded distributions $\P_Z$, for which the norm $|Z|$ and the directional part $Z/|Z|$ of $Z$ are independent and the right tail of the distribution of $|Z|$ decays like a power law. In this case, they showed a (non-Gumbel) limit distribution of $M_n$ that can be described in terms of a suitably defined Poisson point process.
Finally, \cite{Demichel2014} considered unbounded elliptical distributions of the form
$
Z = TAW,
$
where $T$ is a positive and unbounded random variable, $A$ is an invertible $(d \times d)$-dimensional matrix, and $W$ is uniformly distributed on the sphere $\mathcal{S}^{d-1} = \left\{ z \in \R^d: |z| = 1\right\}.$ In this case, the asymptotical behavior of $M_n$ depends on the right tail of the distribution function of $T$ and the multiplicity $k \in \left\{ 1,\ldots,d\right\}$ of the largest eigenvalue of $A.$ In that work, it was assumed that $T$ lies in the max-domain of attraction of the Gumbel law. If the matrix $A$ has a single largest eigenvalue, \cite{Demichel2014} derives a limit law for $M_n$ that can be represented in terms of two independent Poisson point processes on $\R^d$. On the other hand, if $A$ has a multiple largest eigenvalue and $T$ satisfies an additional technical assumption, $M_n$ has a Gumbel limit law. If $k=d$, the random vector $Z$ has a spherically symmetric distribution, and their result is the same as that stated by \cite{RaoJanson2015}.

If $\P_{Z}$ has a bounded support, \cite{Lao2010} and \cite{MayerMol2007} deduced a Weibull limit distribution for $M_n$  in a very general setting if the distribution of $Z$ is supported by the $d$-dimensional unit ball $\B^d$ for $d \ge 2$.
Furthermore, \cite{Lao2010} obtained limit laws for $M_n$ if $\P_{Z}$ is uniform or non-uniform in the unit square, uniform in regular polygons, or uniform in the $d$-dimensional unit cube, $d \ge 2$.
\cite{Appel2002} obtained a convolution of two independent Weibull distributions as limit law of $M_n$
if $Z$ has a uniform distribution in a planar set with
unique major axis and `sub-$\sqrt{x}$ decay' of its boundary at the endpoints. The latter property is \emph{not} fulfilled if $\P_Z$ is supported by a proper ellipse $E$.
In that case, \cite{Appel2002} were able to derive bounds for the limit law of $M_n$ if $Z$ has a uniform distribution.
The exact limit behavior of $M_n$ if $\P_{Z}$ is uniform in an ellipse has been an open problem for many years. Without giving a proof, \cite{RaoJanson2015} stated that $n^{2/3}(2-M_n)$ has a limit distribution (involving two independent Poisson processes) if $Z$ has a uniform distribution in a proper ellipse with major axis of length $2$. \cite{Schrempp2015} described this limit distribution in terms of two independent sequences of random variables, and  \cite{Schrempp2016} generalized the result of \cite{RaoJanson2015} to the case that $\P_Z$ is uniform or non-uniform over a $d$-dimensional ellipsoid.
Being more precise, the underlying set $E$ in \cite{Schrempp2016} is
$$
E = \left\{z \in \R^d:  \left( \frac{z_1}{a_1} \right)^2 + \left( \frac{z_2}{a_2} \right)^2 + \ldots + \left(
\frac{z_d}{a_d} \right)^2 \le 1\right\},
$$
where $d\ge2$ and $a_1 > a_2 \ge a_3 \ge \ldots \ge a_d > 0$. Since $a_1 > a_2$, the ellipsoid $E$ has a unique major axis of length $2a_1$ with `poles' $(a_1,0,\ldots,0)$ and $(-a_1,0,\ldots,0)$. If the distribution $\P_Z$ is supported by such a set $E$ and $\P_Z(E \cap O)>0$ for each neighborhood $O$ of each of the two poles, the unique major axis makes sure that the asymptotical behavior of $M_n$ is determined solely by the shape of $\P_Z$ close to these poles. \cite{Schrempp2016} investigated distributions $\P_{Z}$ with a Lebesgue density $f$ on $E$, so that $f$ is continuous and bounded away from $0$ near the poles. Hence, the uniform distribution on $E$ was a special case of that work.
It turned out that $2a_1 - M_n$ has to be scaled by the factor $n^{2/(d+1)}$  to obtain a non-degenerate limit distribution. In order to show this weak convergence, a related setting had been considered, in which the random points are the support of a specific series of Poisson point processes $\mathbf{Z}_n$ in $E$.  Writing $\diam(\mathbf{Z}_n)$ for the diameter of the support of $\mathbf{Z}_n$, it turned out that $n^{2/(d+1)}(2a_1-\diam(\mathbf{Z}_n ))$ has a limiting distribution involving two independent Poisson processes that live on a subset $P$ of $\R^d$, the shape of which is determined by $a_1,\ldots,a_d$.
By use of the so-called de-Poissonization technique, $n^{2/(d+1)}(2a_1-M_n)$ has the same limit distribution as $n$ tends to infinity.\\

From the proofs given in \cite{Schrempp2016}, it is quite obvious that only the values of the density at the poles and the curvature of the boundary $\partial E$ of $E$ at the poles determine the limiting distribution of $n^{2/(d+1)}(2a_1-M_n)$, but \emph{not} the fact that $E$ is an ellipsoid. The latter observation was the starting point for this work:
Our main result is a generalization of the result stated in \cite{Schrempp2016} to distributions that are supported by a $d$-dimensional set $E$, $d \ge 2$, with `unique diameter' of length $2a>0$ between the poles $(-a,0,\ldots,0)$ and $(a,0,\ldots,0)$ and a smooth boundary at the poles. The formal assumptions on $E$ are stated in
\autoref{sec_main_results}.
If the density $f$ of $Z$ on $E$ is continuous and bounded away from $0$ close to the poles,  $n^{2/(d+1)}(2a-\diam(\mathbf{Z}_n ))$ has a non-degenerate limiting distribution also in this setting. Again, this limit law involves two independent Poisson processes that live on potentially different subsets $P_\ell$ and $P_r$ of $\R^d$.
%The connection between these two sets and the original set $E$ can be described as follows:
The shape of $P_\ell$ is only determined by the principal curvatures and the corresponding principal curvature directions of $\partial E$ at the left pole $(-a,0,\ldots,0)$. The same holds true for $P_r$ and the right pole $(a,0,\ldots,0)$.

The paper is organized as follows.
In \autoref{sec_fundamentals} we will fix our general notation, and
\autoref{sec_main_results} contains our assumptions and our main result, which is \autoref{thm_main_result}, the proof of  which will be given in \autoref{sec_proof_main_thm}.
\autoref{sec_generalizations_unique_diameter} contains several generalizations of the main result for underlying sets with a `unique diameter'. These include more general distributions $\P_Z$, a limit theorem for the joint convergence of the $k$ largest distances among $Z_1,\ldots,Z_n$ and $p$-norms and so-called `$p$-superellipsoids', where $1 \le p < \infty $.
\autoref{sec_generalizations_no_unique_diameter} deals with generalizations of our main result to settings where $E$ does not have a `unique diameter', and it concludes with a fundamental open problem concerning
Pearson Type II distributions that are supported by an ellipsoid with at least two but less than $d$ major half-axes.

%%%%%%%%%%%%%%%%%%%%%%%%%%%%%%%%%%%%%%%%%%%%%%%%%%%%%%%%%%%%%%%%%%%%%%%%%%%%%%%%%%%%%%%%%%%%%%%%%%%%%%%
%%%%%%%%%%%%%%%%%%%%%%%%%%%%%%%%%%%%%%%%%%%%%%%%%%%%%%%%%%%%%%%%%%%%%%%%%%%%%%%%%%%%%%%%%%%%%%%%%%%%%%%
%%%%%%%%%%%%%%%%%%%%%%%%%%%%%%%%%%%%%%%%%%%%%%%%%%%%%%%%%%%%%%%%%%%%%%%%%%%%%%%%%%%%%%%%%%%%%%%%%%%%%%%

\section{Fundamentals}
\label{sec_fundamentals}
Throughout, vectors are understood as column vectors, but if there is no danger of misunderstanding, we write them -- depending on the context -- either as row or as column vectors.
We use the abbreviation $\widetilde z := (z_2,\ldots,z_d)$ for a point $z = (z_1,\ldots,z_d) \in \R^d$. Given a function $s:\R^{d-1} \to \R, \widetilde z \mapsto s(\widetilde z)$, let $s_j(\widetilde z)$ denote the partial derivative of $s$ with respect to the component $z_j$ for $j \in \left\{ 2,\ldots,d\right\}$. Notice that, for instance, $s_2$ stands for the partial derivative of $s$ with respect to $z_2$, \emph{not} with respect to the second component of $\widetilde z$. The gradient $\big(s_2(\widetilde z),\ldots,s_d(\widetilde z)\big)$ of $s$ at the point $\widetilde z$ will be denoted by $\nabla s(\widetilde z)$.
Likewise $s_{ij}(\widetilde z)$ is the second-order partial derivative with respect to $z_i$ and $z_j$.
Without stressing the dependence on the dimension, we write $\mathbf{0}$ for the origin in $\R^{i}$ and $\mathbf{e}_j$ for the $j$-th unit vector in $\R^{i}$ for $ i,j \in \N := \left\{ 1,2,\ldots\right\}$ with $j \le i$.
The scalar product of $x,y \in \R^i$ will be denoted by $\langle x,y\rangle$, $i \in \N$.
For a subset $ A \subset \R^{d}$ and $c > 0$ we write $c \cdot A := \left\{ c \cdot z: z \in A\right\}$,
and we put $\R_{+}:= [0,\infty)$. Furthermore, $m_d$ stands for $d$-dimensional Lebesgue measure, and the $i$-dimensional identity matrix will be denoted by $\mathrm{I}_i$, $i \in \N$.
Each unspecified limit refers to $n \to \infty$, and for two real-valued sequences $(a_n)_{n \in \N}$ and $(b_n)_{n \in \N}$, where $b_n \ne 0$ for each $n \in \N$, we write $a_n \sim b_n$ if $a_n/b_n \to 1$.
For a density $g$, a measure $\mu$ on $\R^d$ and a Borel set $A \in \mathcal{B}^d $ we put $g\big|_A(z) := g(z)$ if $z \in A$ and $0$ otherwise, and write $\mu\big|_A(B) := \mu(A \cap B)$ if $B \in \mathcal{B}^d $.
Convergence in distribution and equality in distribution will be denoted by $\overset{\mathcal{D}}{\longrightarrow }$ and $\overset{\mathcal{D}}{= } $, respectively. The components of a random vector $Z_i$ are given by $Z_i = (Z_{i,1},\ldots,Z_{i,d})$ for $i \ge 1$.
Finally, we write $N \overset{\mathcal{D}}{= } \Po(\lambda) $ if the random variable $N$ has a Poisson distribution with parameter $\lambda >0 $.

Regarding point processes, we mainly adopt the notation of~\cite{Resnick1987}, Chapter 3. A point process $\xi$ on some space $D$, equipped with a $\sigma$-field $\mathcal{D} $, is a measurable map from some probability space $(\Omega,\mathcal{A},\P )$ into $\big(M_p(D),\mathcal{M}_p(D)\big) $, where $M_p(D)$ is the set of all point measures $\chi$ on $D$, equipped with the smallest $\sigma$-field $\mathcal{M}_p(D)$ rendering the evaluation maps $\chi \mapsto \chi(A)$ from $M_p(D) \to [0,\infty]$ measurable for all $A \in \mathcal{D}.$
We call the point process $\xi$ simple if
$
\P\big(\xi(\left\{ z\right\}) \in \left\{ 0,1\right\} \text{ for all } z \in D\big) = 1.
$
A Poisson process with intensity measure $\mu$ is a point process $\xi$ satisfying
\begin{equation}
\label{eq_def_Poisson_process}
  \P\big(\xi(A) = k \big ) =
  \begin{cases}
    e^{-\mu(A)}\frac{\mu(A)^k}{k!}, & \text{if } \mu(A) < \infty,\\
    0, &  \text{if } \mu(A) = \infty,
  \end{cases}
\end{equation}
for $A \in \mathcal{D} $ and $ k \in \N\cup \left\{ 0\right\}$.
Moreover, $\xi(A_1),\ldots,\xi(A_i)$ are independent for any choice of $ i \ge 2$ and mutually disjoint sets $A_1,\ldots,A_i\in \mathcal{D}$. We briefly write $\xi \overset{\mathcal{D}}{= }  \PRM(\mu)$. If $\xi$ is a Poisson process with intensity measure $\mu$, \eqref{eq_def_Poisson_process} means $\xi(A) \overset{\mathcal{D}}{= } \Po\big(\mu(A)\big)$ and hence $\E\big[\xi(A)\big] = \mu(A)$ for $A \in \mathcal{D}$.
According to Corollary 6.5 in~\cite{Last2017}, for each Poisson process $\xi$ on $D$ there is a sequence $\mathcal{X}_1,\mathcal{X}_2,\ldots  $ of random points in $D$ and a $\left\{ 0,1,\ldots,\infty \right\}$-valued random variable $N$ so that
$$
\xi = \sum_{i=1}^{N}\varepsilon_{\mathcal{X}_i }, \quad \text{almost surely.}
$$
Because of this property we use the notation $ \xi = \left\{ \mathcal{X}_i, i \ge 1\right\} $, whenever  $\xi$ is a simple Poisson process and $\xi(D) = \infty$ almost surely. This terminology is motivated by the notion of a point process as a random set of points.

We will use the bold letters $\mathbf{X},\mathbf{Y}  $ and $\mathbf{Z} $ to denote point processes, and the convention will be as follows: Point processes supported by the whole underlying set $E$ will get a name involving the letter $\mathbf{Z}$. In contrast, the letter $\mathbf{X} $ always stands for processes that live only on the left half $E \cap \left\{ z_1 \le 0\right\}$ of $E$ and $\mathbf{Y} $ for those that are supported by the right half $E \cap \left\{ z_1 \ge 0\right\}$ of $E$. This distinction will be very useful to shorten the notation.
If, for instance, $\mathbf{X} = \left\{ \mathcal{X}_i, i \ge 1 \right\}$ is a point process on $\R^d$, we write $\mathcal{X}_i  = (\mathcal{X}_{i,1},\ldots,\mathcal{X}_{i,d} )$ to denote the coordinates of $\mathcal{X}_{i}. $ We finally introduce a very special sequence of Poisson processes: If $Z_1,Z_2,\ldots$ are i.i.d. with common distribution $\P_Z$, and  $N_n$ is independent of this sequence and has a Poisson distribution with parameter $n$, then
$$
\mathbf{Z}_n :=   \sum_{i=1}^{N_n}\varepsilon_{Z_{i}}, \quad n \in \N,
$$
is a Poisson process in $\R^d$ with intensity measure $n\P_{Z}$, and we have
$$
\diam(\mathbf{Z}_n ) =  M_{N_n} = \max_{1 \le i,j \le N_n} \big|Z_i - Z_j\big|.
$$

\section{Conditions and main results}
\label{sec_main_results}
Our basic assumption on the shape and the orientation of the underlying set $E$ is that its finite diameter is attained by exactly one pair of points, both of which lie on the $z_1$-axis. Being more precise, we assume the following:
\vspace{2mm}
\begin{condition}
\label{cond_unique_diameter}
Let $E \subset \R^d$ be a closed subset with $0 < 2a = \diam(E) < \infty  $ and $(-a,\mathbf{0} ), (a,\mathbf{0}  )\in E$. Furthermore, we assume
\begin{equation}
\label{eq_cond_unique_diameter}
|x-y| < 2a \qquad \text{for each} \qquad (x,y)\in \big( E\backslash \left\{ (-a,\mathbf{0} ), (a,\mathbf{0} )\right\} \big) \times E.
\end{equation}
\end{condition}
\vspace{2mm}

Speaking of  a `unique diameter', we will always mean that the underlying set satisfies \autoref{cond_unique_diameter}.
The two points $(-a,\mathbf{0} ),(a,\mathbf{0} )\in E$ are henceforth called the `poles' of $E$. There is no loss of generality in assuming that the poles of $E$ are given by $(-a,\mathbf{0} )$ and $(a,\mathbf{0} )$.
For every set having a diameter of length $2a >0$ we can find a suitable coordinate system so that this assumption is satisfied.
Since we will consider distributions with $m_d$-densities supported by $E$, it will be no loss of generality either that we assume $E$ to be closed.
So, condition~\eqref{eq_cond_unique_diameter} guarantees that $M_n$ will be determined by two points lying close to $(-a,\mathbf{0} )$ and $(a,\mathbf{0} )$, respectively, at least for large $n$ and a suitable distribution $\P_Z$.
Our assumption on the shape of $E$ close to both poles is as follows:
\vspace{2mm}
\begin{condition}
\label{cond_shape_pole_caps}
There are constants $\delta_\ell,\delta_r \in (0,a]$, open neighborhoods $O_\ell,O_r\subset \R^{d-1}$ of $\mathbf{0} \in \R^{d-1} $ and twice continuously differentiable functions $s^\ell : O_\ell \to \R_{+}$, $s^r : O_r \to \R_{+}$, so that \begin{align}
  E_\ell :=\ &E \cap \left\{ z_1 < -a + \delta_\ell\right\}
  =\  \left\{  (z_1,\widetilde z) \in \R^d: -a + s^\ell(\widetilde z) \le z_1 < -a + \delta_{\ell} , \widetilde z\in O_\ell\right\}
  \label{eq_def_E_l}
  \intertext{and}
  E_r :=\ & E \cap \left\{a - \delta_r < z_1\right\}
  =\  \left\{ (z_1,\widetilde z) \in \R^d: a - \delta_r < z_1 \le  a - s^r(\widetilde z) , \widetilde z\in O_r\right\}.
  \label{eq_def_E_r}
\end{align}
\end{condition}

Since $(-a,\mathbf{0} ),(a,\mathbf{0} )\in E$, we have $s^\ell(\mathbf{0} ) = s^r(\mathbf{0} ) = 0$, and we write $H_i$ for the Hessian of $s^{i}$ at the point $\mathbf{0} $. In view of the unique diameter of $E$ between $(-a,\mathbf{0} )$ and $(a,\mathbf{0} )$, we know the following facts about $\nabla s^{i}(\mathbf{0} )$ and $H_i$, $i \in \left\{ \ell,r\right\}$:
\begin{lemma}
\label{lem_first_part_deriv_are_0_and_Hessian_pos_def}
For $i \in\left\{ \ell,r\right\}$ we have $\nabla s^i(\mathbf{0}) = \mathbf{0} $. Furthermore, the matrix $H_i$ is symmetric and positive definite, and all $d-1$ eigenvalues of $H_i$ are larger than $1/2a$.
\end{lemma}
The proof of this lemma can be found in \autoref{sec_appendix}.
According to \autoref{lem_first_part_deriv_are_0_and_Hessian_pos_def}, the matrices $H_\ell$ and $H_r$ are orthogonally diagonalizable and all eigenvalues, denoted by $\kappa_2^i \le \ldots \le \kappa_{d}^i$, $i \in \left\{ \ell,r\right\}$, in ascending order, are real-valued and positive. The subscripts $2,\ldots,d$ instead of $1,\ldots,d-1$ are chosen deliberately. Because of the very close connection between these eigenvalues and the components $z_2,\ldots,z_d$ in our main theorem, this notation is much more intuitive for our purposes. See especially the end of this section for an illustration of the aforementioned connection.
For $i \in \left\{ \ell,r\right\}$ we choose an orthonormal basis $\left\{ \mathbf{u}_2^i,\ldots,\mathbf{u}_d^i  \right\}$ of $\R^{d-1}$, consisting of corresponding eigenvectors; namely $H_i \mathbf{u}_j^i =\kappa_j^i  \mathbf{u}_j^i $ for $j \in \left\{ 2,\ldots,d\right\}$. Putting $U_i := (\mathbf{u}_2^i\ |\ \ldots\ |\ \mathbf{u}_d^i  )$, we have $U_iU_i^\top = \mathrm{I}_{d-1}$ and $U_i^\top H_i U_i = \diag(\kappa_2^i,\ldots,\kappa_d^i) =: D_i.$ \\

It is quite obvious that \autoref{cond_unique_diameter} restricts the possible Hessians $H_\ell$ and $H_r$.
It would be desirable to find a one-to-one relation between the unique diameter of $E$ assumed in \autoref{cond_unique_diameter} on the one hand and all possible Hessians $H_\ell$ and $H_r$ on the other hand.
But describing this relation in its whole generality would be technically very involved. Fortunately, we can state a simple but still very general condition on the Hessians to guarantee that $E \cap \left\{ |z_1| > a - \delta\right\}$ has a unique diameter between $(-a,\mathbf{0} )$ and $(a,\mathbf{0} )$ for $\delta >0$ sufficiently small. Unless otherwise stated we will always study sets fulfilling the following condition:
\vspace{3mm}
\begin{condition}
\label{cond_A_eta_pos_semi_definite}
For some constant $\con \in (0,1)$, the $2(d-1) \times 2(d-1)$-dimensional matrix
$$
A(\con) :=
\begin{pmatrix}
  2a\con D_\ell - \mathrm{I}_{d-1} & U_\ell^\top U_r \\
  U_r^\top U_\ell & 2a\con D_r - \mathrm{I}_{d-1}
\end{pmatrix}
$$
is positive semi-definite.
\end{condition}
\vspace{3mm}

We will briefly write $A(\con) \ge 0$ to denote this property.
Notice that $A(\con_1) \ge 0 $ implies $A(\con_2) \ge 0$ for each $\con_2 > \con_1$ since $D_\ell$ and $D_r$ are diagonal matrices with positive entries on their main diagonals.
Due to the fact that $D_\ell,D_r,U_\ell$ and $U_r$ depend only on the curvature of $\partial E$ at the poles, \autoref{cond_A_eta_pos_semi_definite} is obviously not sufficient to ensure \eqref{eq_cond_unique_diameter} (figuring in \autoref{cond_unique_diameter}) for the whole set $E$. But \autoref{lem_suff_con_unique_diam} will show that \autoref{cond_A_eta_pos_semi_definite} guarantees that \eqref{eq_cond_unique_diameter} holds true for $E$ replaced with $E \cap \left\{ |z_1| > a - \delta\right\}$ and $\delta >0 $ sufficiently small. This assertion can be interpreted as `\autoref{cond_A_eta_pos_semi_definite} ensures the unique diameter of $E$ close to the poles'.
Focussing on sets satisfying \autoref{cond_A_eta_pos_semi_definite} will be no strong limitation in the following sense:
If $A(1)$ is \emph{not} positive semi-definite, then $E$ cannot have a unique diameter between the poles, see \autoref{lem_A_1_has_to_be_pos_def}. Hence, the only relevant case not covered by \autoref{cond_A_eta_pos_semi_definite} is given by
$$
A(1)\ge0, \quad \text{ but } \quad A(\con) \ngeq 0 \quad   \text{ for each } \quad  \con \in (0,1).
$$
At first sight, \autoref{cond_A_eta_pos_semi_definite} looks quite technical. A much more intuitive and sufficient, but \emph{not} necessary condition for \autoref{cond_A_eta_pos_semi_definite} to hold is
\begin{equation}
\label{eq_suff_cond_princ_curv}
  \frac{1}{\kappa_2^\ell}+ \frac{1}{\kappa_2^r} < 2a,
\end{equation}
see \autoref{lem_suff_cond_princ_curv}. We may thus check \autoref{cond_unique_diameter} (at least close to the poles) for many sets by merely looking at the smallest eigenvalues of $H_\ell$ and $H_r$.
Now that we have stated our conditions on the underlying set $E$, we can focus on distributions supported by $E$.
In this section we consider distributions $\P_{Z}$ with a Lebesgue density $f$ on $E$ satisfying the following property of continuity at the poles:
\vspace{2mm}
\begin{condition}
\label{cond_density_p_l_p_r}
Let $f: E \to \R_{+}$ with $\int_{E}f(z)\,\mathrm{d}z =1 $. We further assume that $f$ is continuous at the poles $(-a,\mathbf{0} )$, $(a,\mathbf{0} )$ with
$$
p_\ell := f(-a,\mathbf{0} ) > 0 \qquad \text{and} \qquad p_r := f(a,\mathbf{0} ) > 0.
$$
\end{condition}
\vspace{2mm}

Defining the `pole-caps of length $\delta$' via
\begin{equation}
\label{eq_def_E_l_delta_E_r_delta}
E_{\ell,\delta} := E_\ell \cap \left\{ - a \le z_1 \le -a + \delta\right\}
\quad \text{and} \quad
E_{r,\delta} := E_r \cap \left\{ a - \delta \le z_1 \le a\right\}
\end{equation}
for $0<\delta < \min\left\{ \delta_\ell,\delta_r\right\}$, the property of continuity assumed in \autoref{cond_density_p_l_p_r} can be rewritten as $f(z) = p_i \big(1 + o(1)\big)$, where $o(1)$ is uniformly on $E_{i,\delta}$ as $\delta \to 0$, $i \in \left\{ \ell,r\right\}$.
Now, we only need one more definition before we can formulate our main result.
Putting
\begin{equation}
\label{eq_def_P_H}
 P(H) := \left\{ (z_1,\widetilde z) \in \R^d:  \frac{1}{2} \widetilde z^\top  H \widetilde z \le z_1 \right\}
\end{equation}
for some $(d-1)\times (d-1)$-dimensional matrix $H$, the set $P(H_\ell)$ (resp. $P(H_r)$) describes the shape of $E$ near the left (resp. right) pole if we `look through a suitably distorted magnifying glass', see \autoref{lemma_transformation_density} for details.
The boundaries of $P(H_\ell)$ and $P(H_r)$ are elliptical paraboloids. Now we are prepared to state our main result.

\begin{theorem}
\label{thm_main_result}
If Conditions~\ref{cond_unique_diameter} to \ref{cond_density_p_l_p_r} hold, then
\begin{equation}
\label{eq_theorem_main_result}
n^{\frac{2}{d+1}}\big(2a - \mathrm{diam}(\mathbf{Z}_n )\big) \overset{\mathcal{D}}{\longrightarrow } \min_{i,j \ge 1} \left\{ \mathcal{X}_{i,1}  + \mathcal{Y}_{j,1} - \frac{1}{4a}\big|\widetilde  {\mathcal{X}}_{i}- \widetilde {\mathcal{Y}}_{j}\big|^2 \right\},
\end{equation}
where $\left\{ \mathcal{X}_i, i \ge 1\right\} \overset{\mathcal{D}}{= }  \PRM\big(p_\ell\cdot m_d\big|_{P(H_\ell)}\big) $ and $\left\{ \mathcal{Y}_j, j \ge 1\right\} \overset{\mathcal{D}}{= }  \PRM\big(p_r\cdot m_d\big|_{P(H_r)}\big) $ are independent Poisson processes. The same holds true if we replace $ \mathrm{diam}(\mathbf{Z}_n )$ with $M_n$.
\vspace{1mm}
\end{theorem}
\vspace{1mm}
\begin{proof}
See \autoref{sec_proof_main_thm}.
\qed
\end{proof}

A special case of this result is given if we assume that $E$ is a proper ellipsoid. The following corollary illustrates that \autoref{thm_main_result} is a generalization of the main result in~\cite{Schrempp2016}:
\begin{corollary}
\label{cor_Ellipsoid}
Let $a_1 > a_2 \ge a_3 \ge \ldots \ge a_d > 0$, and put
$$
E := \left\{z \in \R^d:   \sum_{j=1}^{d}\left(\frac{z_j}{a_j} \right)^2  \le 1\right\}.
$$
The values $a_1,\ldots,a_d$ are called the `half-axes' of the ellipsoid $E$. Obviously, this set
has a unique diameter of length $2a_1$ between the points $(-a_1,\mathbf{0} )$ and $(a_1,\mathbf{0} )$; i.e. \autoref{cond_unique_diameter} holds true with $a = a_1$.
Putting $\delta_\ell := \delta_r := a_1$,
$$
O_\ell := O_r := \left\{\widetilde z \in \R^{d-1}:  \sum_{j=2}^{d}\left(\frac{z_j}{a_j} \right)^2 < 1\right\}
\qquad
\text{and}
\qquad
s^\ell(\widetilde z) := s^r(\widetilde z) := a_1 - a_1\sqrt{1 - \sum_{j=2}^{d}\left(\frac{z_j}{a_j} \right)^2 },
$$
\autoref{cond_shape_pole_caps} is fulfilled, too.
Some easy calculations show that the Hessians $H_\ell$ and $H_r$ of $s^{\ell}$ and $s^{r}$ at $\mathbf{0}$ are given by
\begin{align*}
H_\ell = H_r = \diag\left( \frac{a_1}{a_2^2}\ ,\ \ldots\ ,\ \frac{a_1}{a_d^2}\right).
\end{align*}
This means that the eigenvalues of both $H_\ell$ and $H_r$ are $\kappa_j^i = a_1/a_j^2$.
Since $ a_2 < a_1$, we have
$$
\frac{1}{\kappa_2^\ell}+ \frac{1}{\kappa_2^r} = \frac{a_2^2}{a_1} + \frac{a_2^2}{a_1} = 2\frac{a_2^2}{a_1} = 2a_1 \left( \frac{a_2}{a_1}\right)^2 < 2a_1 = 2a.
$$
Hence, inequality \eqref{eq_suff_cond_princ_curv} holds true and thus \autoref{cond_A_eta_pos_semi_definite} is fulfilled. With
\begin{align*}
 P(H_\ell) = P(H_r)
 &= \left\{ z \in \R^d: \frac{1}{2} \widetilde z^\top H_{\ell}\widetilde z \le  z_1\right\}
= \left\{ z \in \R^d: \frac{1}{2} \sum_{j=2}^{d} \frac{a_1}{a_j^2}\cdot z_j^2 \le  z_1\right\}
 = \left\{ z \in \R^d: \sum_{j=2}^{d} \left( \frac{z_j}{a_j} \right)^2  \le  \frac{2z_1}{a_1}\right\},
\end{align*}
we can apply \autoref{thm_main_result} for distributions in $E$ satisfying \autoref{cond_density_p_l_p_r}, in accordance with Theorem 2.1 in \cite{Schrempp2016}.
\end{corollary}

\begin{corollary}
\label{cor_ellipse_uniform}
If $Z$ has a uniform distribution in the ellipsoid $E$ given in \autoref{cor_Ellipsoid}, \autoref{cond_density_p_l_p_r} holds true with
$$
p_\ell := p_r := \frac{1}{m_d(E)} = \left( \frac{\pi^{\frac{d}{2}}}{\Gamma\left( \frac{d}{2}+1\right)}\prod_{i=1}^{d}a_i  \right)^{-1} > 0.
$$
Hence, \autoref{thm_main_result} is applicable. In the special case $d=2$ and $a_1 = 1$ we have $a_2 < 1$,
$p_\ell = p_r = 1/(\pi a_2)$,
$$
P:=P(H_\ell) = P(H_r)= \left\{ z \in \R^2:  \left( \frac{z_2}{a_2} \right)^2 \le  2z_1\right\},
$$
 and it follows that
\begin{equation}
\label{eq_Limit_Dist_Ellipse}
n^{2/3}(2 - M_n ) \overset{\mathcal{D}}{\longrightarrow } \min_{i,j \ge 1} \left\{ \mathcal{X}_{i,1}  + \mathcal{Y}_{j,1} - \frac{1}{4} (\mathcal{X}_{i,2}-\mathcal{Y}_{j,2})^2 \right\},
\end{equation}
with two independent Poisson processes $\mathbf{X} = \left\{ \mathcal{X}_i, i \ge 1\right\} \overset{\mathcal{D}}{= } \PRM\big(p_\ell\cdot m_2\big|_P\big) $ and $\mathbf{Y} = \left\{ \mathcal{Y}_j, j \ge 1\right\} \overset{\mathcal{D}}{= } \PRM\big(p_r\cdot m_2\big|_P\big) $.
\end{corollary}
To illustrate the speed of convergence in \autoref{cor_ellipse_uniform}, we present the result of a simulation study. To this end, define $G(x,y) := x_1 + y_1 - (x_2 - y_2)^2/4$. In the proof of \autoref{lemma_continuity_of_G_hat} one can see that $G(x,y) \ge c(x_1+y_1)$, $(x,y) \in P(H_\ell) \times P(H_r)$, for some fixed $c \in (0,1)$. Therefore, the probability that a point $\mathcal{X}_i $ with a `large' first component $\mathcal{X}_{i,1} $ determines the minimum above is `small' (we omit details). The same holds for $\mathcal{Y}_j$. We can thus approximate the limiting distribution above by taking independent Poisson processes with intensity measures $p_\ell\cdot m_2\big|_{\widetilde P}$ and $p_r\cdot m_2\big|_{\widetilde P}$ where $\widetilde P := P(H_\ell) \cap \left\{ z \in \R^2: z_1 \le b\right\}$ for some fixed $b >0$.
The larger the minor half-axis $a_2$ is
(i.e. the more $E$ becomes `circlelike'), the larger $b$ has to be chosen in order to have a good approximation of the distributional limit in
 \eqref{eq_Limit_Dist_Ellipse} (we omit details). See \autoref{fig_plot_ellipse_simulation} for an illustration of the sets $E$ (left) and $P$ (right) and \autoref{fig_cdf_plot_ellipse_uniform} for the result of a simulation. Notice the different scalings between the left-~and the right-hand image in \autoref{fig_plot_ellipse_simulation}.
\begin{figure}[ht]
	\centering
    \includegraphics[trim = 80 570 110 100, clip ]
    {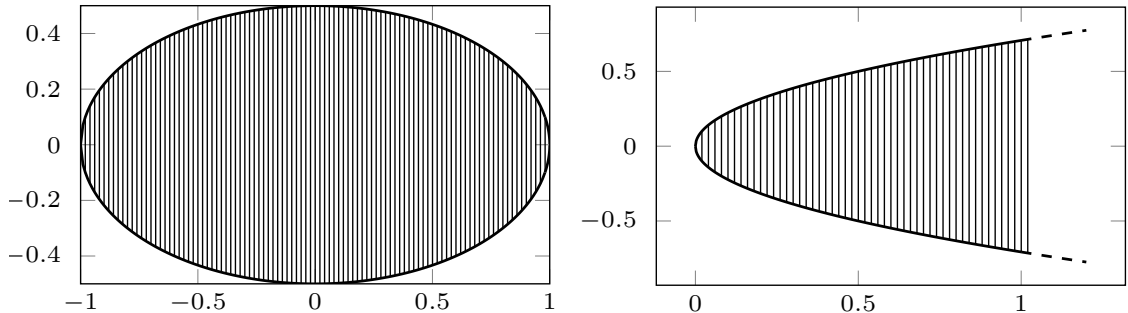}
	\caption{The sets $E$ (left) and $P$ (right)  in the setting of \autoref{cor_ellipse_uniform} for $d=2$ with $a_1=1, a_2=1/2$.}
	\label{fig_plot_ellipse_simulation}
\end{figure}

\begin{figure}[ht]
	\centering
    \includegraphics[trim = 120 530 100 100, clip ]
    {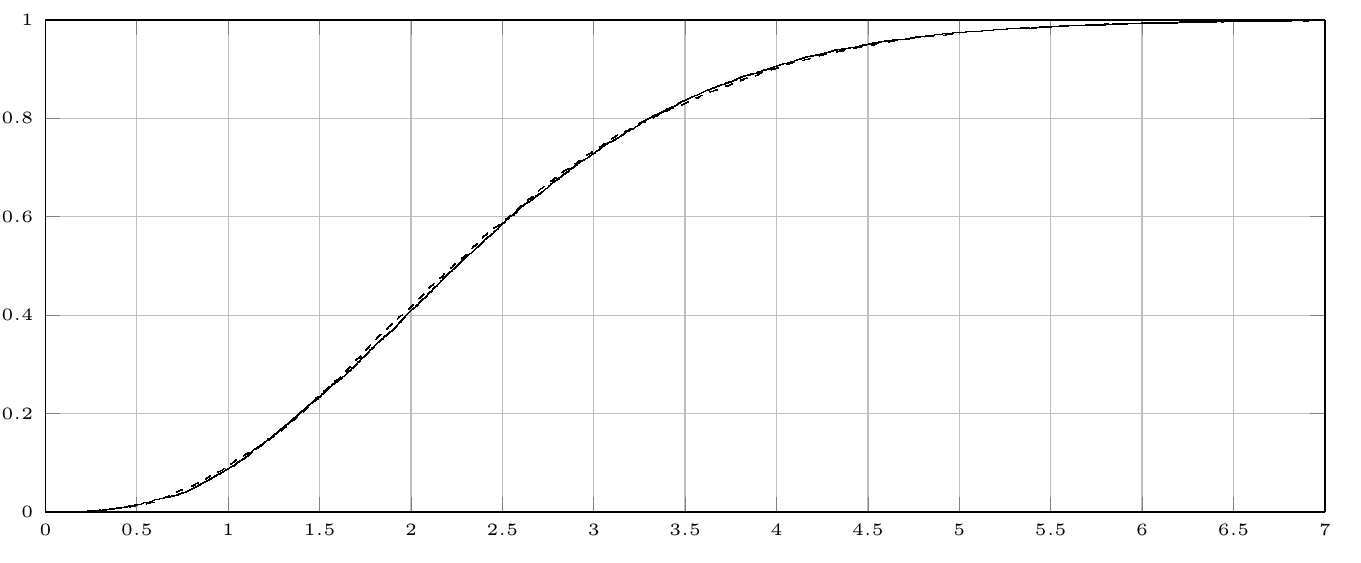}
	\caption{Empirical distribution function of $n^{2/3}(2-M_n)$ in the setting of \autoref{cor_ellipse_uniform} for $d=2$ with $a_1=1, a_2=1/2$, $n=1000$ (solid, 5000 replications). The limit distribution is approximated as described after \autoref{cor_ellipse_uniform} with $b = 10$ (dashed, 5000 replications).
}
	\label{fig_cdf_plot_ellipse_uniform}
\end{figure}

\begin{remark}
\label{rem_interpretation_as_hypersurfaces}
The `outer boundaries' of $E_\ell$ and $E_r$, denoted by
$$
M_\ell := \left\{(z_1,\widetilde z) \in \R^d: z_1 = -a+s^\ell(\widetilde z) , \widetilde z \in O_\ell\right\}
\qquad
\text{and}
\qquad
 M_r := \left\{(z_1,\widetilde z) \in \R^d: z_1 = a-s^r(\widetilde z) , \widetilde z \in O_r\right\},
$$
can be interpreted as images of the two hypersurfaces
$$
\mathbf{s}^\ell:\ O_\ell \to \R^d,\ \mathbf{s}^\ell(\widetilde z):=\big(-a + s^\ell(\widetilde z)\ ,\ \widetilde z \ \big)
\qquad
\text{and}
\qquad
\mathbf{s}^r:\ O_r \to \R^d,\ \mathbf{s}^r(\widetilde z):=\big(a -s^r(\widetilde z)\ ,\ \widetilde z \ \big).
$$
For $ i \in \left\{ \ell,r\right\}$ the first partial derivatives of $\mathbf{s}^i $ are given by
$$
\mathbf{s}_2^i(\widetilde z) =
    \Big(
    \ s_2^i(\widetilde z) \ , \    1  \ , \   0   \ , \  \ldots  \ , \  0 \
    \Big)
, \quad \ldots\quad , \
\mathbf{s}_d^i(\widetilde z) =
   \Big(
    \ s_d^i(\widetilde z) \ , \    0  \ , \  \ldots\ , \   0  \ , \  1 \
    \Big) .
$$
These $d-1$ vectors are linearly independent for each $\widetilde z \in O_i$, which means that the hypersurfaces $\mathbf{s}^\ell $ and $\mathbf{s}^r $  are regular, see Definition 3.1.2 in~\cite{Csikos2014}. From \autoref{lem_first_part_deriv_are_0_and_Hessian_pos_def} we further know $\mathbf{s}_j^i(\mathbf{0} ) = \mathbf{e}_j$ for $i \in \left\{ \ell,r\right\}$ and each $j \in \left\{ 2,\ldots,d\right\}$. Hence, the two unit normal vectors of the hypersurface $\mathbf{s}^{i} $ at the pole $\mathbf{s}^{i}(\mathbf{0} ) $ are given by $\pm \mathbf{e}_1 $.
Looking at Appendix A.2.2 in \cite{Schrempp2017} -- especially its ending -- we know (because of $\nabla s^i(\mathbf{0} )= \mathbf{0}$) that the eigenvalues $\kappa_j^i$ of the Hessian $H_i$ are exactly the principal curvatures of the hypersurface $\mathbf{s}^i $ at the pole $\mathbf{s}^i (\mathbf{0} )$ with respect to the unit normal vector $\mathbf{e}_1$ if $i = \ell$ and $-\mathbf{e}_1$ if $i = r$, respectively.
We can further conclude that
\begin{equation*}
\label{eq_Connection_v_and_u}
 \mathbf{v}_j^i :=
 \begin{pmatrix}
   0 \\
   \mathbf{ u}_j^i
 \end{pmatrix} \in \R^d
\end{equation*}
are the corresponding principal curvature directions.
Using the notation introduced after \autoref{lem_first_part_deriv_are_0_and_Hessian_pos_def} and some easy transformations show
\begin{align*}
  P(H_i) &= \left\{ z_1 \mathbf{e}_1 + \sum_{j=2}^{d}z_j \mathbf{v}_j^i
  \in \R^d:  \frac{1}{2}  \sum_{j=2}^{d}\kappa_j^i z_j^2 \le z_1\right\}, \label{eq_represenatation_P_Hl_by_princ_dir_1}
\end{align*}
see \cite[p. 30]{Schrempp2017} for more details.
This representation is sometimes called the `normal representation of the osculating paraboloid $P(H_i)$', and it justifies the notation of the principal curvatures $\kappa_j^i$ with indices $2,\ldots,d$ instead of $1,\ldots,d-1$. If we have $\mathbf{v}_{j}^i = \mathbf{e}_j  $ for each $ j \in \left\{ 2,\ldots,d\right\}$, we especially get
\begin{equation*}
\label{eq_represenatation_P_Hl_by_princ_dir_2}
P(H_i)= \left\{ z  \in \R^d:  \frac{1}{2}  \sum_{j=2}^{d}\kappa_j^i z_j^2 \le z_1\right\}.
\end{equation*}
\autoref{cor_Ellipsoid} is a special case of this situation.

\end{remark}

\section{Proof of \autoref{thm_main_result} }
\label{sec_proof_main_thm}
The proof of \autoref{thm_main_result} is divided into three subsections. The first one is mainly devoted to the study of some geometric properties of the set $E$ close to the poles. In \autoref{subsec_proof_main_thoerem_conv_point_processes} we will deal with the convergence of Poisson random measures, which will be crucial for the main part of the proof of \autoref{thm_main_result}, given in \autoref{subsubsec_main_part_proof_main_thm}.

\subsection{Geometric considerations}
\label{subsec_proof_main_thm_geom_considerations}
First of all, note that \autoref{cond_A_eta_pos_semi_definite} holds true if, and only if,
  \begin{equation}
  \label{eq_cond_1_unique_diam}
    0\le 2a\con\left( \alpha^\top D_\ell\alpha + \beta^\top D_r \beta\right)+2\alpha^\top U_\ell^\top U_r\beta-|\alpha|^2 - |\beta|^2, \qquad \text{for all $\alpha,\beta \in \R^{d-1}$},
  \end{equation}
and since $\kappa_2^i$ and $\kappa_d^i$ are the smallest and the largest eigenvalues of $H_i$, the so-called min--max theorem by Courant--Fischer yields
\begin{equation}
\label{eq_Courant_Fischer}
  \kappa_2^{i}| \widetilde z|^2 \le \widetilde z^\top  H_i \widetilde z \le \kappa_d^{i} | \widetilde z|^2
\end{equation}
for each $\widetilde z \in \R^{d-1}$ and $i \in \left\{ \ell,r\right\}$.
In view of \autoref{lem_first_part_deriv_are_0_and_Hessian_pos_def} it is clear that the second-order Taylor series expansions of $s^{i}$ at the point $\mathbf{0} $ is
$$
s^i(\widetilde z) =  \frac{1}{2}\widetilde z^\top  H_i  \widetilde  z + R_i (\widetilde z),
$$
where $R_i (\widetilde z) = o\big(|\widetilde z|^2 \big)$ and $i \in \left\{ \ell,r\right\}$. From \eqref{eq_def_E_l} and \eqref{eq_def_E_r} we obtain the representations
\begin{align}
  E_\ell
  &= \left\{  (z_1,\widetilde z) \in \R^d: -a + \frac{1}{2}\widetilde z^\top  H_\ell \widetilde  z + R_\ell (\widetilde z) \le z_1 < -a + \delta_{\ell} , \widetilde z\in O_\ell\right\} \label{eq_representation_E_l_Taylor}
  \intertext{and}
  E_r
  &= \left\{ (z_1,\widetilde z) \in \R^d: a - \delta_r < z_1 \le  a - \frac{1}{2}\widetilde z^\top  H_r \widetilde  z - R_r (\widetilde z) , \widetilde z\in O_r\right\},
  \label{eq_representation_E_r_Taylor}
\end{align}
which will be widely used throughout this work.
Now we need some additional definitions.
We shift the set $E_\ell$ to the right by $a\cdot\mathbf{e}_1 $ along the $z_1$-axis and call this set $P_1(H_\ell)$. The set $E_r$ will be translated by $-a\cdot\mathbf{e}_1 $ along the $z_1$-axis to the left, and it will then be reflected at the plane $\left\{ z_1 = 0\right\}$. We call the resulting set $P_1(H_r)$. Looking at \eqref{eq_representation_E_l_Taylor} and~\eqref{eq_representation_E_r_Taylor}, we have
\begin{equation}
\label{eq_def_P_1_H_i}
P_1(H_i) = \left\{  (z_1,\widetilde z) \in \R^d: \frac{1}{2}\widetilde z^\top   H_i \widetilde  z + R_i (\widetilde z) \le z_1 < \delta_i , \widetilde z\in O_i\right\}
\end{equation}
for $i \in \left\{ \ell,r\right\}$. The reason underlying this construction will be seen later in~\eqref{eq_reason_first_power_bigger}.
In addition to $P_1(H_i)$,
we introduce the constant
\begin{equation}
\label{eq_def_eta_hat}
\widehat \con := \frac{1+ \con^{-1}}{2},
\end{equation}
based on the constant $\con \in (0,1)$ from \autoref{cond_A_eta_pos_semi_definite}. The subsequent remark will point out two very important properties of $\widehat \con$, that will be essential for the proofs to follow:

\begin{remark}
\label{rem_properties_con_hat}
 Since $\con \in (0,1)$, we have $\widehat \con > 1$, and it follows that  $P(H_i) \varsubsetneq \widehat \con \cdot P(H_i)$ for $i \in \left\{ \ell,r\right\}$. Without this technical expansion of the limiting sets $P(H_i) $, several proofs would become much more complicated. The second important property is that $\widehat \con$ is not `too large' in the sense that
 $$
 1 - \con\widehat \con = 1 - \con\frac{1+ \con^{-1}}{2} = 1 - \frac{\con + 1}{2} = \frac{1 - \con}{2} > 0.
 $$
 This inequality will be crucial for the proofs of \autoref{lemma_R_is_little_o_of_G_tilde} and \autoref{lemma_continuity_of_G_hat}.
\end{remark}

As stated in \autoref{rem_properties_con_hat}, we will need the set $\widehat \con \cdot P(H_i)$ for $i \in \left\{ \ell,r\right\}$. For later use, we give a more convenient representation of these sets:
\begin{remark}
\label{rem_representation_widehat_xi_P_H_i}
For $i \in \left\{ \ell,r\right\}$ we obtain from \eqref{eq_def_P_H}
\begin{align*}
  \widehat \con \cdot P(H_i)
  &= \left\{ \widehat \con  \cdot z \in \R^d: z \in P(H_i)\right\} \\
  &= \left\{z\in \R^d: \widehat \con ^{-1} z \in P(H_i)\right\}\\
  &= \left\{z\in \R^d: \frac{1}{2}\left( \widehat \con ^{-1} \widetilde z \right)^\top  H_i \left( \widehat \con ^{-1} \widetilde z \right)  \le \widehat \con ^{-1}z_1\right\} \\
  &= \left\{z\in \R^d: \frac{1}{2}\widetilde z^\top  H_i\widetilde z  \le \widehat \con z_1\right\}.
\end{align*}
\end{remark}
In the following, we have to consider simultaneously points $x$, that are lying close to the left pole, and points $y$, lying close to the right one. For this purpose, we use the definitions of the pole-caps $E_{\ell,\delta}$ and $E_{r,\delta}$ given in \eqref{eq_def_E_l_delta_E_r_delta} and put $E_{\delta} := E_{\ell,\delta} \times E_{r,\delta}$ to yield
\begin{equation}
\label{eq_Representation_E_delta}
  E_{\delta} = \big\{ (x,y) \in E_\ell \times E_r: -a \le x_1 \le -a + \delta, a - \delta \le y_1 \le a\big\}.
\end{equation}

The next lemma shows the reason for introducing the sets $\widehat \con \cdot P(H_i)$, $i \in \left\{ \ell,r\right\}$. The inclusion stated there will be crucial for the proof of the subsequent \autoref{lemma_R_is_little_o_of_G_tilde} and for the main part of the proof of \autoref{thm_main_result} itself.
\begin{lemma}
\label{lem_shifted_poles_in_xi_hat_P_H_i}
There is some constant $\delta^* \in \big(0, \min\left\{ \delta_\ell,\delta_r\right\}\big]$, so that
the inclusion
\begin{equation}
\label{eq_shifted_poles_in_xi_hat_P_H_i_proof_help_0}
\big( P_1(H_\ell) \cap \left\{ z_1 \le \delta\right\}\big) \times
\big( P_1(H_r) \cap \left\{ z_1 \le \delta\right\}\big)
\subset \widehat \con \cdot P(H_\ell) \times \widehat \con \cdot P(H_r)
\end{equation}
holds true for each $\delta \in (0,\delta^{*}]$. In other words, we have
\begin{align}
\frac{1}{2}\widetilde x^\top  H_\ell \widetilde x \le  \widehat \con (a+x_1) \qquad \text{and}\qquad
\frac{1}{2}\widetilde y^\top  H_r \widetilde y \le  \widehat \con (a-y_1)
\label{eq_shifted_poles_in_xi_hat_P_H_i_proof_help_1}
\end{align}
for all $(x,y) \in E_{\delta^{*}}$.
\end{lemma}
\begin{proof}
Observe \autoref{rem_representation_widehat_xi_P_H_i} and the construction of $P_1(H_\ell)$ and $P_1(H_r)$ at the beginning of this section for checking the equivalence between \eqref{eq_shifted_poles_in_xi_hat_P_H_i_proof_help_0} and \eqref{eq_shifted_poles_in_xi_hat_P_H_i_proof_help_1}.
Without loss of generality we only show the first inequality of~\eqref{eq_shifted_poles_in_xi_hat_P_H_i_proof_help_1} for $(x,y) \in E_{\delta}$ and $\delta > 0$ sufficiently small.
If $\delta < \delta_\ell$, it follows from~\eqref{eq_representation_E_l_Taylor} and the definition of $E_\delta$ that
$$
x \in   \left\{  (z_1,\widetilde z) \in \R^d: -a + \frac{1}{2}\widetilde z^\top   H_\ell \widetilde  z + R_\ell (\widetilde z) \le z_1 \le -a + \delta , \widetilde z\in O_l\right\},
$$
whence
\begin{equation}
\label{eq_shifted_poles_in_xi_hat_P_H_i_proof_help_2}
\frac{1}{2}\widetilde x^\top   H_\ell \widetilde  x + R_\ell (\widetilde x) \le a +  x_1 \le \delta.
\end{equation}
As $\delta \to 0$ we get $|\widetilde x| \to 0$ on $E_{\delta}$, and
because of $R_\ell (\widetilde x) = o\left( |\widetilde x|^2\right)$ the relation $R_\ell (\widetilde x) = o\left( \widetilde x^\top  H_\ell \widetilde x\right)$ holds true, too.
Putting $\varepsilon := \frac{\con^{-1}-1}{\con^{-1}+1}>0$, we obtain for sufficiently small $\delta >0$
$$
\big|R_\ell (\widetilde x)\big| \le \frac{\varepsilon}{2}\widetilde x^\top   H_\ell \widetilde  x
$$
for every $(x,y) \in E_{\delta}$. Combining this inequality with \eqref{eq_shifted_poles_in_xi_hat_P_H_i_proof_help_2}
shows that
$$
\frac{1-\varepsilon}{2}\widetilde x^\top   H_\ell \widetilde  x  \le a + x_1
$$
and hence, by the definition of $\widehat \con$ given in \eqref{eq_def_eta_hat},
\begin{align*}
\frac{1}{2}\widetilde x^\top   H_\ell \widetilde  x
&\le \frac{1}{1-\varepsilon}(a + x_1)%\\
%&= \frac{1}{1-\frac{\con^{-1}-1}{\con^{-1}+1} }(a+x_1)\\
%&= \frac{1}{\frac{\con^{-1}+1-\con^{-1}+1}{\con^{-1}+1}}(a+x_1)\\
%&= \frac{1 + \con^{-1}}{2}(a+x_1)\\
= \widehat \con (a+x_1).
\end{align*}
Choosing $\delta^{*}$ in such a way that both inequalities figuring in \eqref{eq_shifted_poles_in_xi_hat_P_H_i_proof_help_1} hold true for each $(x,y) \in E_{\delta^{*}}$ finishes the proof.
\qed
\end{proof}
In the following, we will, without loss of generality, only investigate $E_{\delta}$ for $\delta \in (0,\delta^{*}]$ to ensure the validity of \eqref{eq_shifted_poles_in_xi_hat_P_H_i_proof_help_1}.\\

In the next step we examine the behavior of $|x-y|$ for $x$ close to the left pole of $E$ and $y$ close to the right one. For this purpose, we consider $\R^{2d}$ to describe the simultaneous convergence of $x$ to the left pole of $E$ and $y$ to the right pole. Some straightforward calculations show that
the second-order Taylor polynomial of $(x,y) \mapsto |x-y|$ at the point $\mathbf{a} := (-a,\mathbf{0} , a,\mathbf{0} ) \in \R^{2d}$ is given by
$
 -x_1 + y_1 + \frac{1}{4a}|\widetilde x - \widetilde y|^2.
$
As $(x,y )\to \mathbf{a} = (-a,\mathbf{0} , a,\mathbf{0} )$, we obtain
\begin{equation}
\label{eq_taylor_eukl_dist_at_poles}
|x-y| = - x_1 + y_1 +\frac{1}{4a}|\widetilde x- \widetilde y|^2 + R(x,y),
\end{equation}
where $R(x,y) = o\big( |(x,y) - \mathbf{a} |^2\big)$,
uniformly on the ball of radius $r$ and center $\mathbf{a} $ as $r \to 0$. This uniform convergence holds especially on
$E_{\delta}$  (given in \eqref{eq_Representation_E_delta})
as $\delta \to 0$.
Putting
\begin{equation*}%\label{eq_Def_G_Schlange}
  \widetilde G(x,y) := (a + x_1) + (a - y_1) - \frac{1}{4a}|\widetilde x- \widetilde y|^2,
\end{equation*}
we infer
\begin{equation}
\label{eq_2a_minus_norm_is_G_tilde_minus_R}
2a - |x-y| = \widetilde G(x,y) - R(x,y).
\end{equation}

\begin{lemma}
\label{lemma_R_is_little_o_of_G_tilde}
We have $R(x,y) = o\big( \widetilde  G(x,y)\big)$, uniformly on $E_{\delta}$ as $\delta \to 0$.
\end{lemma}
\begin{proof}
Notice that
\begin{equation}
\label{eq_proof_lemma_uniformly}
  \frac{R(x,y)}{\widetilde G(x,y)} = \frac{R(x,y)}{|(x,y) - \mathbf{a} |^2} \cdot\frac{|(x,y) - \mathbf{a}|^2}{\widetilde G(x,y)}
  = o(1)\frac{|(x,y) - \mathbf{a}|^2}{\widetilde G(x,y)}
\end{equation}
as $\delta \to 0$, where $o(1)$ is uniformly on $E_{\delta}$. It remains to show that $|(x,y) - \mathbf{a}|^2 / \widetilde G(x,y)$
is bounded on $E_{\delta}$ for small $\delta > 0$. Assume without loss of generality that $|x_1| \le |y_1| < a$. In view of $x \in E_\ell$ and $y \in E_r$, we get $0 < a - y_1 \le a + x_1$.
Consider in a first step the numerator of the right-most fraction figuring in~\eqref{eq_proof_lemma_uniformly}. With \eqref{eq_Courant_Fischer} and \autoref{lem_shifted_poles_in_xi_hat_P_H_i} we obtain for $(x,y) \in E_\delta$ and sufficiently small $\delta>0$
\begin{align*}
  |(x,y) - \mathbf{a}|^2
  &=  (a+x_1)^2 + (a-y_1)^2 + |\widetilde x|^2 + |\widetilde y|^2\\
  &\le  (a+x_1)^2 + (a-y_1)^2 + \frac{1}{\kappa_2^\ell}\widetilde x^\top  H_\ell \widetilde x + \frac{1}{\kappa_2^r}\widetilde y^\top  H_r \widetilde y\\
  &\le  (a+x_1)^2 + (a-y_1)^2 + \frac{2 \widehat \con}{\kappa_2^\ell} (a+x_1) + \frac{2 \widehat \con}{\kappa_2^r} (a-y_1)\\
  &\le  (a+x_1)^2 + (a+x_1)^2 + \frac{2 \widehat \con}{\kappa_2^\ell} (a+x_1) + \frac{2 \widehat \con}{\kappa_2^r} (a+x_1)\\
  &= (a+x_1)\left( 2(a+x_1) + \frac{2 \widehat \con}{\kappa_2^\ell} + \frac{2 \widehat \con}{\kappa_2^r}\right).
\end{align*}
As a consequence of $(x,y) \in E_{\delta}$ and $\delta \to 0$ we get $x_1 \to -a$, and thus the term inside the big brackets converges to $\frac{2 \widehat \con}{\kappa_2^\ell}+\frac{2 \widehat \con}{\kappa_2^r}$. We can conclude that there is a constant $c >0$ so that $
|(x,y) - \mathbf{a} |^2 < (a + x_1)\cdot c
$
for every $(x,y) \in E_{\delta}$ and sufficiently small $\delta >0$.
In a second step we look at the denominator of the right-most fraction
figuring in~\eqref{eq_proof_lemma_uniformly}.
Writing $\widetilde x = U_\ell \alpha$ and $\widetilde y = U_r \beta$, we deduce that
\begin{align*}
  \widetilde G(x,y)
  %&= (a + x_1) + (a - y_1) - \frac{1}{4a}|\widetilde x - \widetilde y|^2\\
  &= (a + x_1) + (a - y_1) - \frac{1}{4a}\big(|\widetilde x|^2 + |\widetilde y|^2 - 2\widetilde x^\top \widetilde y\big)\\
  &= (a + x_1) + (a - y_1) - \frac{1}{4a}\big(|\alpha|^2 + |\beta|^2 - 2\alpha^\top U_\ell^\top  U_r \beta\big) .
  \intertext{Inequality \eqref{eq_cond_1_unique_diam} now shows that}
  \widetilde G(x,y)
  &\ge (a + x_1) + (a - y_1) - \frac{1}{4a}2a\con\left( \alpha^\top D_\ell\alpha + \beta^\top D_r \beta\right) \\
  &= (a + x_1) + (a - y_1) - \frac{1}{2}\con\left( \widetilde x^\top  U_\ell D_\ell U_\ell^\top  \widetilde x + \widetilde y^\top  U_r D_r U_r^\top  \widetilde y\right)\\
  &= (a + x_1) + (a - y_1)  - \frac{1}{2}\con\left( \widetilde x^\top  H_\ell\widetilde x + \widetilde y^\top  H_r \widetilde y\right),
  \intertext{and by \autoref{lem_shifted_poles_in_xi_hat_P_H_i} we get for sufficiently small $\delta >0$}
  \widetilde G(x,y)
  &\ge (a + x_1) + (a - y_1)  - \frac{1}{2}\con\big(2\widehat \con (a + x_1) + 2\widehat \con  (a-y_1)\big) \\
  &= (a+x_1)\left( 1 + \frac{a-y_1}{a+x_1} - \con \widehat \con\left( 1 + \frac{a-y_1}{a+x_1}\right)\right)\\
  &= (a+x_1)\left( 1 - \con \widehat \con\right)\left( 1 + \frac{a-y_1}{a+x_1}\right).
\end{align*}
\autoref{rem_properties_con_hat} and $ \frac{a-y_1}{a+x_1} \ge 0$ now yield
$$
  \widetilde G(x,y)
  %&= (a+x_1)\left( 1- \frac{\con+1}{2}\right)\left( 1 + \frac{a-y_1}{a+x_1}\right)\\
  \ge (a+x_1)  \frac{1-\con}{2}\left( 1 + \frac{a-y_1}{a+x_1}\right)
  \ge (a+x_1)  \frac{1-\con}{2},
$$
where $\frac{1-\con}{2}> 0$.
Putting both parts together, we have
\begin{align*}
 \frac{|(x,y) - \mathbf{a} |^2 }{\widetilde G(x,y)}
  \le   \frac{(a+x_1)\cdot c}{(a+x_1)\cdot \frac{1-\con}{2}}
  = \frac{2c}{1 - \con}
\end{align*}
for every $(x,y) \in E_{\delta}$ and $\delta>0$ small enough, and the proof is finished.
\qed
\end{proof}

\subsection{Convergence of Poisson random measures}
\label{subsec_proof_main_thoerem_conv_point_processes}
In this subsection we will focus on the convergence of Poisson processes inside the sets $P_1(H_i)$ for $i \in \left\{ \ell,r\right\}$. \autoref{lemma_Conv_of_V_n} will be the key to describe the asymptotical behavior of those points of $\mathbf{Z}_n $ lying close to one of the poles if we `look through a suitably distorted magnifying glass' and let $n$ tend to infinity.
In what follows, put
\begin{equation}
\label{eq_def_of_nu}
\nu := \frac{1}{d+1}
\end{equation}
and
$$
T_n(z) := \left(\  n^{2\nu}z_1\ ,\ n ^{\nu}\widetilde z\ \right)
$$
for $n \in \N $ and $ z = (z_1,\widetilde z) \in \R^d.$

\begin{lemma}
\label{lemma_transformation_density}
Suppose that, for $i \in \left\{ \ell,r\right\}$, the random vector $\ZVhelp = (\ZVhelp_1,\ldots,\ZVhelp_d)$ has a density $g$ on $P_1(H_i) \cap \left\{ z_1 \le \delta^*\right\}$
with $g(z) = p\big(1+o(1)\big)$ uniformly on $P_1(H_i) \cap \left\{ z_1 \le \delta\right\}$ as $\delta \to 0$ for some $p >0$. Then, for every bounded Borel set $\boundSet \subset \R^d$, we have $\P\big( T_n(\ZVhelp) \in \boundSet \big)= \kappa_n(\boundSet)/n$
with
$ \kappa_n(\boundSet) \to p \cdot m_d\big|_{P(H_i)}(\boundSet)$ as $n \to \infty $.
\end{lemma}
\begin{proof}
To emphasize the support of $g$, we write $g(z)\ind \big\{ z \in P_1(H_i)\cap \left\{ z_1 \le \delta^*\right\}\big\} $ instead of $g(z)$. The Jacobian of $T_n$ is given by
$$
\Delta T_n(x) = \text{det}
\big( \diag\left( n^{2\nu},n^{\nu},\ldots,n^{\nu}\right)
\big)
=n^{(d+1)\nu} = n,
$$
and therefore the random vector $T_n(\ZVhelp)$ has the density
\begin{align*}
 g_n(z) =&\; \frac{g\left( T_n^{-1}(z)\right)}{n}
  = \;  \frac{1}{n} g\left( \frac{z_1}{n^{2\nu}}, \frac{1}{n^{\nu}}\widetilde z\right) \ind\big\{ z \in P_n(H_i)\big\},
\end{align*}
where $P_n(H_i) := T_n\big(P_1(H_i)\cap \left\{ z_1 \le \delta^*\right\}\big)$. In view of \eqref{eq_def_P_1_H_i} we get
\begin{align*}
  P_n(H_i)&= \left\{ z \in \R^d: T_n^{-1}(z) \in P_1(H_i)\cap \left\{ z_1 \le \delta^*\right\}\right\}\\
  &= \left\{ z \in \R^d:  \frac{1}{2}\left( \frac{1}{n^{\nu}}\widetilde z\right)^\top  H_i \left( \frac{1}{n^{\nu}}
  \widetilde  z\right) + R_i \left( \frac{1}{n^{\nu}}\widetilde z\right) \le \frac{z_1}{n^{2\nu}} \le \delta^*, \frac{1}{n^{\nu}}\widetilde z\in O_i\right\}\\
  &= \left\{ z \in \R^d:  \frac{1}{2}\widetilde z^\top  H_i \widetilde z + n^{2\nu}R_i \left( \frac{1}{n^{\nu}}\widetilde z\right)\le z_1 \le n^{2\nu}\delta^*, \widetilde z\in n^{\nu}O_i\right\}.
\end{align*}
Since $O_i$ is an open neighborhood of $\mathbf{0} \in \R^{d-1} $
and
$$
n^{2\nu}R_i \left(\frac{1}{n^{\nu}}\widetilde z\right)
= |\widetilde z|^2\cdot \frac{R_i \left(\frac{1}{n^{\nu}}\widetilde z\right)}{\left|\frac{1}{n^{\nu}}\widetilde z\right|^2}
 \to 0
$$
as $n \to \infty $ for each fixed $\widetilde z \in \R^{d-1}$,
we see that $\ind\big\{ z \in P_n(H_i)\big\} \to \ind\big\{ z \in P(H_i)\big\}$ for almost all $z \in \R^d$. Observe that this convergence does not hold true for $z = (z_1,\widetilde z) \in \R^d$ with $\frac{1}{2}\widetilde z^\top H_i \widetilde z = z_1$ and $R_i \big(\frac{1}{n^{\nu}}\widetilde z\big) > 0$ for infinitely many $n \in \N$. But, since $\left\{ z \in \R^d: \frac{1}{2}\widetilde z^\top H_i \widetilde z = z_1\right\}$ has Lebesgue measure $0$, these points will have no influence on the integrals to follow.
For each Borel set $\boundSet \subset \R^d$, we have
\begin{align*}
   \P\big( T_n(\ZVhelp) \in \boundSet \big)
  &=  \int_{\boundSet} g_n(z)\,\mathrm{d}z
  = \frac{1}{n} \int_{\boundSet} g\left( \frac{z_1}{n^{2\nu}}, \frac{1}{n^{\nu}}\widetilde z\right) \ind\left\{ z \in P_n(H_i)\right\}\,\mathrm{d}z.
\end{align*}
If $\boundSet$ is bounded, $\sup\left\{ z_1 : (z_1,\widetilde z) \in \boundSet\right\} \le i_1$ for some $i_1 \in [0,\infty )$. Consequently, $\left( \frac{z_1}{n^{2\nu}}, \frac{1}{n^{\nu}}\widetilde z\right) \in \left\{t \in \R^d: t_1 \le \frac{i_1}{n^{2\nu}}\right\}$ for every $z \in \boundSet$.
Since $g(z) = p\big(1+o(1)\big)$, uniformly on $P_1(H_i) \cap \left\{ z_1 \le \delta\right\}$ as $\delta \to 0$, we obtain
$g\left( \frac{z_1}{n^{2\nu}}, \frac{1}{n^{\nu}}\widetilde z\right) = p\big(1+o(1)\big)$
uniformly on $\boundSet$ as $n \to \infty$, whence
\begin{align*}
  \P\big( T_n(\ZVhelp) \in \boundSet \big)
  &= \frac{1}{n} \cdot p\int_{\boundSet} \big(1 + o(1)\big) \ind\left\{ z \in P_n(H_i)\right\}\,\mathrm{d}z =: \frac{1}{n} \cdot \kappa_n(\boundSet).
\end{align*}
Since $\boundSet$ is bounded and $\big(1 + o(1)\big)\ind\left\{ z \in P_n(H_i)\right\} \to \ind\left\{ z \in P(H_i)\right\}$ for almost all $z \in \R^d$, the dominated convergence theorem gives
\begin{align*}
 \lim_{n\to\infty}  \kappa_n(\boundSet)
  &= p \int_{\boundSet} \lim_{n\to\infty} \big(1 + o(1)\big) \ind\left\{ z \in P_n(H_i)\right\}\,\mathrm{d}z
  = p \int_{\boundSet} \ind\left\{ z \in P(H_i)\right\}\,\mathrm{d}z
  = p\cdot m_d\big|_{P(H_i)}(\boundSet).
\end{align*}
\qed
\end{proof}

\begin{remark}
\label{rem_inclusion_limiting_set}
In the main part of the proof of \autoref{thm_main_result} in \autoref{subsubsec_main_part_proof_main_thm}, we will have to investigate point processes living inside the sets $P_1(H_i)$. But, contrary to the setting in~\cite{Schrempp2016}, the inclusion $P_1(H_i) \subset P(H_i)$ does \emph{not} hold in general, and hence especially \emph{not} $P_n(H_i) \subset  P(H_i)$ for every $n \ge 1$. Therefore, the set
$P(H_i)$ is in general not suitable as state space for our point processes.
Letting $\R^d$ be the state space would rectify this problem, but then the proof of \autoref{lemma_continuity_of_G_hat} would fail. So, this is the point where it becomes crucial to slightly enlarge the sets $P(H_i)$ via $\widehat \con \cdot P(H_i)$.
According to \eqref{eq_shifted_poles_in_xi_hat_P_H_i_proof_help_0} and the choice of $\delta^{*}$ we have
\begin{equation}
\label{eq_Inclusion_in_eta_hat_P_H_i}
 P_1(H_i) \cap \left\{ z_1 \le \delta^{*}\right\}
\subset \widehat \con \cdot P(H_i)
\end{equation}
for $i \in \left\{ \ell,r\right\}$. If $z \in \widehat \con \cdot P(H_i)$, then $T_n(z) = \left( n^{2\nu}z_1,n^{\nu}\widetilde z\right)$ and \autoref{rem_representation_widehat_xi_P_H_i} yield
$$
\frac{1}{2}(n^{\nu}\widetilde z)^\top  H_i (n^{\nu}\widetilde z) = n^{2\nu}\frac{1}{2}\widetilde z^\top  H_i \widetilde z \le \widehat \con n^{2\nu} z_1,
$$
i.e, we have $T_n(z) \in \widehat\con \cdot P(H_i)$
for every $n \ge 1$. We thus get the inclusion
$$
T_n\big( \widehat\con \cdot P(H_i)\big) \subset \widehat\con \cdot P(H_i)
$$
for each $n \ge 1$,
and \eqref{eq_Inclusion_in_eta_hat_P_H_i} implies
$$
T_n\big( P_1(H_i) \cap \left\{ z_1 \le \delta^{*}\right\}\big) \subset \widehat\con \cdot P(H_i).
$$
Thus, we can use the state space $\widehat\con \cdot P(H_\ell)$ for the point processes representing the random points near the left pole and $\widehat\con \cdot P(H_r)$ for the corresponding processes near the right pole. In the proofs to follow, it will be very important to consider only the sets $E_\delta$ $($given in \eqref{eq_Representation_E_delta}$)$ with $\delta \in (0,\delta^{*}]$. Without this restriction, the point processes could `leave' their state space, and the proof of \autoref{lemma_continuity_of_G_hat} would fail. Since the asymptotical behavior of the maximum distance will be determined close to the poles, this restriction does not mean any loss of generality.
Without \autoref{cond_A_eta_pos_semi_definite} it could be very complicated to find state spaces that are large enough to include the processes $($close to the poles$)$ but are also small enough to allow an adapted version of \autoref{lemma_continuity_of_G_hat}. These state spaces would have to be defined depending on $($the signs of$)$ the error functions $R_i $ in every direction of $\R^{d-1}$, we omit details.
\end{remark}

As before, let $\ZVhelp = (\ZVhelp_1,\ldots,\ZVhelp_d)$ have a density $g$ on $P_1(H_i) \cap \left\{ z_1 \le \delta^{*}\right\}$ with $g(z) = p\big(1+o(1)\big)$ uniformly on $P_1(H_i) \cap \left\{ z_1 \le \delta\right\}$ as $\delta \to 0$ for some $p >0$. For $n \in \N$ and some fixed $c > 0$ let $\widetilde {\mathbf{\ZVhelp}}_n$ be a Poisson process with intensity measure $nc\cdot \P_\ZVhelp$. With independently chosen
$N_n \overset{\mathcal{D}}{= } \Po(nc)$ and i.i.d. $\widetilde Z_1, \widetilde Z_2, \ldots$ with distribution $\P_\ZVhelp$, we have
$
\widetilde {\mathbf{\ZVhelp}}_n \overset{\mathcal{D}}{= }  \sum_{j=1}^{N_n}\varepsilon_{\widetilde Z_{j}}.
$
According to the Mapping Theorem for Poisson processes, see \cite[p. 38]{Last2017}, $\mathbf{\ZVhelp}_n := \widetilde {\mathbf{\ZVhelp}}_n \circ T_n^{-1}$ is a Poisson process with intensity measure $ \mu_n := nc \cdot \P_\ZVhelp\circ T_n^{-1}$, and the representation above yields
$
\mathbf{\ZVhelp}_n \overset{\mathcal{D}}{= }  \sum_{j=1}^{N_n}\varepsilon_{T_n(\widetilde Z_{j})}.
$
We have $\mathbf{\ZVhelp}_n \in M_p\Big( T_n\big(P_1(H_i)\cap \left\{ z_1 \le \delta^*\right\}\big)\Big)$, $n \in \N$, and because of \autoref{rem_inclusion_limiting_set} it follows that $\mathbf{\ZVhelp}_n \in M_p\big( \widehat \con \cdot P(H_i)\big)$.

\begin{lemma}
\label{lemma_Conv_of_V_n}
Let $\mathbf{\ZVhelp}_n $ be defined as above. Then $\mathbf{\ZVhelp}_n \overset{\mathcal{D}}{\longrightarrow } \mathbf{\ZVhelp}$ with $\mathbf{\ZVhelp} \overset{\mathcal{D}}{= } \PRM(\mu)$ and $\mu := p c \cdot m_d\big|_{P(H_i)} $.
\end{lemma}
\begin{proof}
We use Proposition 3.22 in~\cite{Resnick1987}. Writing $\mathcal{I}$ for the set of finite unions of bounded open rectangles, we have to show that $\P\left( \mathbf{\ZVhelp} (\partial I) = 0\right) = 1$, and that both the conditions (3.23) and~(3.24) in~\cite{Resnick1987}  hold for every $I \in \mathcal{I} $. Because of $\mu(\partial I) = 0$, the first requirement obviously holds, and an application of \autoref{lemma_transformation_density} gives
$$
\mu_n(I) = nc \cdot \left( \P_\ZVhelp\circ T_n^{-1} \right)(I) = nc \cdot \P \big( T_n(\ZVhelp) \in I \big) = c\kappa_n(I) \to \mu(I).
$$
Since $\mathbf{\ZVhelp}_n $ and $\mathbf{\ZVhelp} $ are Poisson processes, we get
$$
\P\big(\mathbf{V}_n(I) = 0 \big) = e^{-\mu_n(I)}\frac{\mu_n(I)^0}{0!} = e^{-\mu_n(I)} \to  e^{-\mu(I)} = e^{-\mu(I)}\frac{\mu(I)^0}{0!} = \P\big(\mathbf{V}(I) = 0 \big)
$$
and
$$
\E\big[ \mathbf{V}_n(I)\big] =  \mu_n(I) \to \mu(I) = \E\big[ \mathbf{V}(I)\big] < \infty .
$$
\qed
\end{proof}

\subsection{Main part of the proof of \autoref{thm_main_result}}
\label{subsubsec_main_part_proof_main_thm}
\begin{proof}%[Proof of Theorem~\ref{thm_main_result}]
As stated before, we only consider $\delta \in (0,\delta^{*}]$.
Recall
$$
E_{\delta} = \big\{ (x,y) \in E_\ell \times E_r: -a \le x_1 \le -a + \delta, a - \delta \le y_1 \le a\big\},
$$
$\delta >0$, and put
$
I_{n}^{\delta} := \big\{ (i,j) : 1 \le  i, j \le N_n, (Z_i,Z_j) \in E_{\delta} \big\},
$
$n \in \N$.
Letting
$
M_{n}^{\delta} := \max_{(i,j) \in I_n^{\delta}} \big|Z_i - Z_j\big|,
$
we obtain $\P\big(M_{n}^{\delta} \neq \diam(\mathbf{Z}_n )\big) \rightarrow  0$ for each $\delta >0$, since both
\begin{align*}
\P\big( Z \in E \cap \left\{ -a \le z_1 \le - a + \delta\right\}\big) > 0
\qquad
\text{and}
\qquad
\P\big( Z \in E \cap \left\{  a - \delta \le z_1 \le a  \right\}\big) > 0
\end{align*}
hold true for each $\delta > 0$.
%$f$ is bounded away from $0$ near the poles.
Hence, it suffices to investigate $M_n^{\delta}$ for some fixed $\delta >0$ instead of $\diam(\mathbf{Z}_n )$.
According to \eqref{eq_2a_minus_norm_is_G_tilde_minus_R} and \autoref{lemma_R_is_little_o_of_G_tilde}, for each $\varepsilon >0$ there is some $\delta >0$ so that
$$
\widetilde G(x,y)(1-\varepsilon) \le 2a - |x-y| \le \widetilde G(x,y)(1+\varepsilon)
$$
for each $(x,y) \in E_{\delta}$. These inequalities imply
\begin{align*}
  (1-\varepsilon)\min_{(i,j) \in I_n^{\delta}}\Big\{ n^{2\nu}\widetilde G(Z_i,Z_j)\Big\}
  \le n^{2\nu}\big( 2a - M_n^{\delta} \big)
 &= \min_{(i,j) \in I_n^{\delta}}\Big\{ n^{2\nu}\big(2a - |Z_i - Z_j|\big)\Big\}
  \le % \min_{(i,j) \in I_n^{\delta}}\Big\{ n^{2\nu}\widetilde G(Z_i,Z_j)(1\pm\varepsilon)\Big\} =
   (1+\varepsilon)\min_{(i,j) \in I_n^{\delta}}\Big\{ n^{2\nu}\widetilde G(Z_i,Z_j)\Big\}.
\end{align*}
Putting $c_{\ell,\delta} := \int_{E_{\ell,\delta}} f(z)\,\mathrm{d}z $ and $c_{r,\delta} := \int_{E_{r,\delta}} f(z)\,\mathrm{d}z $, we define the independent random vectors $ X,Y $ with densities $c_{\ell,\delta}^{-1}f\big|_{ E_{\ell,\delta}}$ and $c_{r,\delta}^{-1}f\big|_{ E_{r,\delta}}$, respectively. Furthermore, for $n \in \N$, we introduce the independent Poisson processes $\widehat{\mathbf{X} }_n$ and $\widehat{\mathbf{Y} }_n$ with intensity measures $nc_{\ell,\delta} \cdot\P_{X}$ and $nc_{r,\delta} \cdot\P_{Y}$, respectively. With independent random elements $N_{\ell,n}, N_{r,n}$, $X_1,X_2,\ldots$, $Y_1,Y_2,\ldots$,
where $N_{\ell,n} \overset{\mathcal{D}}{= } \Po(nc_{\ell,\delta})$, $N_{r,n} \overset{\mathcal{D}}{= } \Po(nc_{r,\delta})$, $X_1,X_2,\ldots $ are i.i.d. with distribution $\P_X$ and $Y_1,Y_2,\ldots $ are i.i.d. with distribution $\P_Y$, we get
$$
\widehat{\mathbf{X} }_n \overset{\mathcal{D}}{= } \sum_{i=1}^{N_{\ell,n}}\varepsilon_{X_i} \qquad \text{and} \qquad
\widehat{\mathbf{Y} }_n \overset{\mathcal{D}}{= } \sum_{j=1}^{N_{r,n}}\varepsilon_{Y_j}.
$$
Letting
$I_n := \big\{ (i,j) : 1 \le i\le N_{\ell,n}, 1 \le j \le N_{r,n}\big\}$,
we obtain
$
M_n^{\delta} \overset{\mathcal{D}}{= } \max_{(i,j) \in I_n} \big|X_i - Y_j\big|.
$
As above, the inequalities
\begin{align}
 (1-\varepsilon)\min_{(i,j) \in I_n}\Big\{ n^{2\nu}\widetilde G(X_i,Y_j)\Big\}
 &\le n^{2\nu}\left( 2a - \max_{(i,j) \in I_n} \big|X_i - Y_j\big| \right)
 \le (1+\varepsilon)\min_{(i,j) \in I_n}\Big\{ n^{2\nu}\widetilde G(X_i,Y_j)\Big\}\label{eq_proof_main_thm_inequalities}
\end{align}
hold, and since $\varepsilon >0$ can be chosen arbitrarily small, it suffices to examine
$
\min_{(i,j) \in I_n}\big\{ n^{2\nu}\widetilde G(X_i,Y_j)\big\}.
$
We get
\begin{align}
 n^{2\nu}\widetilde G(X_i,Y_j)  \
=\ &   n^{2\nu}\left( (a + X_{i,1}) + (a - Y_{j,1}) - \frac{1}{4a}\big|\widetilde X_{i} - \widetilde Y_{j}\big|^2 \right)
=\   G\Big (   n^{2\nu}\big(a + X_{i,1}\big)  ,   n^{\nu}\widetilde X_{i}  ,   n^{2\nu} \big(a - Y_{j,1}\big)  ,   n^{\nu}\widetilde Y_{j}   \Big),\label{eq_reason_first_power_bigger}
\end{align}
where
$$
G: \begin{cases}
  \widehat \con \cdot P(H_\ell) \times \widehat \con \cdot P(H_r) \to \R_{+}, \\
  (x,y) \mapsto x_1 + y_1 - \frac{1}{4a}|\widetilde x- \widetilde y|^2.
\end{cases}
$$
The proof of \autoref{lemma_continuity_of_G_hat} will show that $G(x,y) \ge 0$ for every $(x,y) \in \widehat \con \cdot P(H_\ell) \times \widehat \con \cdot P(H_r) $.
It will be important that $G$ is only defined on $\widehat \con \cdot P(H_\ell) \times \widehat \con \cdot P(H_r) $, not on $\R^{2d}$ (see the proof of \autoref{lemma_continuity_of_G_hat}). This will be no restriction: Because of \autoref{rem_inclusion_limiting_set} it suffices to use instead of $\R^d$ the state spaces $\widehat \con \cdot P(H_\ell)$ and $\widehat \con \cdot P(H_r) $ for the point processes $\mathbf{X}_n $ and $\mathbf{Y}_n $, respectively, where $\mathbf{X}_n $ and $\mathbf{Y}_n $ will be defined later.
To this end, we introduce the Poisson processes
$$
\widetilde{\mathbf{X} }_n := \sum_{i=1}^{N_{\ell,n}}\varepsilon_{( \ a+X_{i,1} \ , \ \widetilde X_{i} \ )} \qquad \text{and} \qquad
\widetilde{\mathbf{Y} }_n := \sum_{j=1}^{N_{r,n}}\varepsilon_{( \ a-Y_{j,1} \ , \ \widetilde Y_{j} \ )}
$$
on $\big(P_1(H_\ell) \cap \{z_1 \le \delta^{*}\}\big) \subset \widehat \con \cdot P(H_\ell)$ and $\big(P_1(H_r) \cap \{z_1 \le \delta^{*}\}\big) \subset \widehat \con \cdot P(H_r)$, respectively.
In view of \autoref{cond_density_p_l_p_r}, we can apply \autoref{lemma_Conv_of_V_n}, and since $\widetilde{\mathbf{X} }_n$ and $\widetilde{\mathbf{Y} }_n$ are independent, we conclude that
\begin{equation}
\label{eq_conv_of_X_and_Y}
  \mathbf{X}_n := \widetilde {\mathbf{X}}_n \circ T_n^{-1} \overset{\mathcal{D}}{\longrightarrow }  \mathbf{X} \qquad \text{and} \qquad
\mathbf{Y}_n := \widetilde {\mathbf{Y}}_n \circ T_n^{-1} \overset{\mathcal{D}}{\longrightarrow } \mathbf{Y}
\end{equation}
on $M_p\big( \widehat \con \cdot P(H_\ell)\big)$ and $M_p\big( \widehat \con \cdot P(H_r)\big)$, respectively, with independent point processes $\mathbf{X} := \left\{ \mathcal{X}_i, i \ge 1\right\} \overset{\mathcal{D}}{= } \PRM\big(p_\ell\cdot m_d\big|_{P(H_\ell)}\big)$ and $ \mathbf{Y} := \left\{ \mathcal{Y}_j, j \ge 1\right\} \overset{\mathcal{D}}{= } \PRM\big(p_r\cdot m_d\big|_{P(H_r)}\big)$. Observe that an application of \autoref{lemma_Conv_of_V_n} to $\mathbf{X}_n$ yields $p = p_\ell/c_{\ell,\delta}$, $c = c_{\ell,\delta}$ and finally $\mu = pc \cdot m_d\big|_{P(H_\ell)} = p_\ell \cdot m_d\big|_{P(H_\ell)}$.
By construction, we have the representations
$$
\mathbf{X}_n = \sum_{i=1}^{N_{\ell,n}}\varepsilon_{T_n( \ a+X_{i,1} \ , \ \widetilde X_{i} \ )} =
\sum_{i=1}^{N_{\ell,n}}\varepsilon_{\left(  \ n^{2\nu}(a + X_{i,1}) \ , \ n^{\nu}\widetilde X_{i} \ \right)}
\qquad
\text{and}
\qquad
\mathbf{Y}_n = \sum_{j=1}^{N_{r,n}}\varepsilon_{T_n( \ a-Y_{j,1} \ , \ \widetilde Y_{j} \ )} =
\sum_{j=1}^{N_{r,n}}\varepsilon_{\left(  \ n^{2\nu}(a - Y_{j,1}) \ , \  n^{\nu}\widetilde Y_{j} \ \right)}.
$$
According to Proposition 3.17 in~\cite{Resnick1987}, $M_p\big( \widehat \con \cdot P(H_\ell)\big)$ and $M_p\big( \widehat \con \cdot P(H_r)\big)$ are separable. By Appendix M10 in \cite{Billingsley1999} we know that $M_p\big( \widehat \con \cdot P(H_\ell)\big)\times M_p\big( \widehat \con \cdot P(H_r)\big)$ is separable, too, and invoking Theorem 2.8 of \cite{Billingsley1999} \eqref{eq_conv_of_X_and_Y} implies $\mathbf{X}_n \times \mathbf{Y}_n \overset{\mathcal{D}}{\longrightarrow } \mathbf{X} \times \mathbf{Y}.$
Define now
\begin{equation*}
\label{eq_def_G_hat}
\widehat G: \begin{cases}
  M_p\big( \widehat \con \cdot P(H_\ell)\big)\times M_p\big( \widehat \con \cdot P(H_r)\big) \to M_p(\R_{+}), \\
  \mu \mapsto \mu \circ G^{-1}.
\end{cases}
\end{equation*}
By construction, we have the representations
\begin{align*}
\widehat G(\mathbf{X}_n \times \mathbf{Y}_n) &= \sum_{i=1}^{N_{\ell,n}}\sum_{j=1}^{N_{r,n}} \varepsilon_{G\left(  \ n^{2\nu}(a + X_{i,1}) \ , \ n^{\nu}\widetilde X_{i} \ , \ n^{2\nu}(a - Y_{j,1}) \ , \  n^{\nu}\widetilde  Y_{j} \ \right) }
\qquad \text{and} \qquad
\widehat G(\mathbf{X} \times \mathbf{Y}) = \sum_{i,j \ge 1}\varepsilon_{G(\mathcal{X}_i,\mathcal{Y}_j  )}.
\end{align*}
Since the mapping $\widehat G$ is continuous (see \autoref{lemma_continuity_of_G_hat}), the continuous mapping theorem gives
\begin{equation}
\label{eq_conv_G_hat_X_times_Y}
  \widehat G(\mathbf{X}_n \times \mathbf{Y}_n) \overset{\mathcal{D}}{\longrightarrow } \widehat G(\mathbf{X} \times \mathbf{Y}).
\end{equation}
For a point process $\xi$ on $\R_+$ we define $t_1(\xi) := \min\left\{ t \ge 0: \xi\big([0,t]\big) \ge 1\right\}.$
The reason for introducing $t_1$ is the very useful relation
$$
\min_{(i,j) \in I_n}\Big\{ n^{2\nu}\widetilde G(X_i,Y_j)\Big\} = t_1\big( \widehat G(\mathbf{X}_n \times \mathbf{Y}_n)\big).
$$
\autoref{lemma_conv_of_t_1} says that
$t_1\big( \widehat G(\mathbf{X}_n \times \mathbf{Y}_n) \big) \overset{\mathcal{D}}{\longrightarrow } t_1\big( \widehat G(\mathbf{X} \times \mathbf{Y}) \big)$
and, because of
\begin{align*}
t_1\big( \widehat G(\mathbf{X} \times \mathbf{Y}) \big)
= \min_{i,j \ge 1} \big\{ G(\mathcal{X}_i,\mathcal{Y}_j  ) \big\}
&= \min_{i,j \ge 1} \left\{ \mathcal{X}_{i,1}  + \mathcal{Y}_{j,1} - \frac{1}{4a} \big| \widetilde {\mathcal{X}}_{i}- \widetilde {\mathcal{Y}}_{j}\big|^2\right\},
\end{align*}
the convergence stated in~\eqref{eq_theorem_main_result} follows from \eqref{eq_proof_main_thm_inequalities} as $\varepsilon \to 0$. Applying Theorem 3.2 in~\cite{MayerMol2007} to the functional $\Psi(\mathbf{Z}_n )= 2 - \diam(\mathbf{Z}_n )$ shows that the same result holds true if we replace $\diam(\mathbf{Z}_n )$ with $M_n$.
\qed
\end{proof}

\begin{remark}
\label{remark_reason_def_T_n}
An explanation for the definition of the rescaling function $T_n(z) = (n^{2\nu}z_1,n^{\nu}\widetilde z)$ with $\nu = 1/(d+1)$ can be found in the proof of \autoref{lemma_Conv_of_V_n}: The $d$ powers of $n$ have to be chosen in such a way that their sum is $1$. This requirement implies $\Delta T_n(z) = n$ in the proof of \autoref{lemma_transformation_density}, whence $\P(T_n(\ZVhelp)\in \boundSet) = \kappa_n(\boundSet)/n$. As seen in the proof of \autoref{lemma_Conv_of_V_n}, the factors $1/n$ and $n$ cancel out, and only $c\kappa_n(\boundSet)$ remains. The reason why the first power is twice the other $d-1$ identical powers is due to the Taylor series expansion of $|x-y|$ in~\eqref{eq_taylor_eukl_dist_at_poles}. This fact fits exactly to the shape of $E$ near the poles, so that $P_n(H_i) = T_n\big(P_1(H_i) \cap \left\{ z_1 \le \delta^{*}\right\}\big)$ can converge to the set $P(H_i)$, $i \in \left\{ \ell,r\right\}$ $($see the proof of \autoref{lemma_transformation_density}$)$. Finally, from~\eqref{eq_reason_first_power_bigger} it is clear that $n^{2\nu}$ is the correct scaling factor.
\end{remark}

We still have to verify the continuity of the function $\widehat G$:

\begin{lemma}
\label{lemma_continuity_of_G_hat}
The function $\widehat G$ is continuous.
\end{lemma}
\begin{proof}
This assertion may be proved in the same way as Proposition 3.18 in~\cite{Resnick1987}. We thus only have to demonstrate that $G^{-1}(K) \subset \widehat \con  \cdot P(H_\ell) \times \widehat \con  \cdot P(H_r)$ is compact if $K \subset \R$ is compact. For this purpose, let $K \subset \R$ be compact. Since $G$ is continuous, $G^{-1}(K)$ is closed, and it remains to show that $G^{-1}(K)$ is bounded. From the specific form of $\widehat \con  \cdot P(H_\ell) \times \widehat \con  \cdot P(H_r)$, $G^{-1}(K)$ can only be unbounded if it is unbounded in $x_1$- or $y_1$-direction (at this point it is important that our state spaces for the point processes are not $\R^d$, but only the subsets $\widehat \con  \cdot P(H_\ell)$ and $\widehat \con  \cdot P(H_r)$).
For fixed $(x,y) \in \widehat\con \cdot P(H_\ell) \times \widehat\con \cdot P(H_r)$, let $\alpha,\beta \in \R^{d-1}$, so that $\widetilde x = U_\ell \alpha$ and $\widetilde y = U_r\beta$. Applying the same transformations as seen for $\widetilde G(x,y)$ in the proof of \autoref{lemma_R_is_little_o_of_G_tilde} to $G(x,y)$ yields
$$
  G(x,y)
  \ge x_1 + y_1 - \con\left(\frac{1}{2} \widetilde x^\top  H_\ell\widetilde x + \frac{1}{2}\widetilde y^\top  H_r \widetilde y\right),
 $$
and using the representation of $\widehat \con \cdot P(H_i)$ given in \autoref{rem_representation_widehat_xi_P_H_i} shows that
$$
  G(x,y)
  \ge x_1 + y_1 - \con\left(\widehat \con  x_1 + \widehat \con  y_1\right)
  = \left( 1 - \con\widehat \con\right)(x_1 + y_1)
  = \frac{1-\con}{2}(x_1 + y_1).
$$
Since $\con \in (0,1)$, we have $\frac{1-\con}{2} > 0$ and the assumption $(x,y) \in \widehat \con  \cdot P(H_\ell) \times \widehat \con \cdot P(H_r)$ implies $(x_1,y_1) \in \R_+^2$, so that $G(x,y) \ge 0$ for each $(x,y) \in \widehat \con  \cdot P(H_\ell) \times \widehat \con  \cdot P(H_r)$.
If $x_1 \to \infty$ and/or $y_1 \to \infty$, the lower bound $ \frac{1-\con}{2}(x_1+ y_1)$ for $G(x,y)$ also tends to infinity. From the boundedness of $K$ it follows that $G^{-1}(K)$ has to be bounded in $x_1$- and $y_1$-direction, too. This argument finishes the proof.
\qed
\end{proof}

Finally, we have to prove the last lemma, used in the proof of \autoref{thm_main_result}:

\begin{lemma}
\label{lemma_conv_of_t_1}
We have $t_1\big(\widehat G(\mathbf{X}_n \times \mathbf{Y}_n  ) \big) \overset{\mathcal{D}}{\longrightarrow } t_1\big(\widehat G(\mathbf{X} \times \mathbf{Y}  ) \big)$.
\end{lemma}
\begin{proof}
 In a first step we will show that $\widehat G(\mathbf{X} \times \mathbf{Y}  ) \big(\left\{ t\right\}\big) = 0$ almost surely for each $t \ge 0$. For this purpose, we consider the set
$$
G^{-1}\big( \left\{ t\right\}\big)
= \left\{ (x,y) \in \widehat \con \cdot P(H_\ell) \times \widehat \con \cdot P(H_r): x_1 + y_1 - \frac{1}{4a}|\widetilde x - \widetilde y|^2 = t\right\}.
$$
For some fixed $y^{*} \in \widehat \con \cdot P(H_r)$ we define
$
A(y^{*}):=   \big\{ x \in \widehat \con \cdot P(H_\ell): (x,y^{*}) \in G^{-1}( \left\{ t\right\})\big\}
$
and obtain
$$
  A(y^{*})
  \phantom{:}=  \ \left\{ x \in \widehat \con \cdot P(H_\ell): x_1 + y_1^{*} - \frac{1}{4a}|\widetilde x - \widetilde y^{*}|^2 = t\right\}
=  \ \left\{ x \in \widehat \con \cdot P(H_\ell): \sqrt{4a\big(x_1 - (t- y_1^{*})\big)} = |\widetilde x - \widetilde y^{*}|\right\}.
$$
Since the set $A(y^{*})$ has Lebesgue-measure $0$, we can conclude that $\mathbf{X}\big(A(y^{*})\big) = 0 $ almost surely for each $y^{*} \in \widehat \con \cdot P(H_r)$. This result implies $\widehat G(\mathbf{X} \times \mathbf{Y}  ) \big(\left\{ t\right\}\big) = 0$ almost surely for each $t \ge 0$.
In the following, we will write $\xi := \widehat G(\mathbf{X} \times \mathbf{Y}  ) $ and $\xi_n :=\widehat G(\mathbf{X}_n \times \mathbf{Y}_n  )$ for $n \in \N$.
In view of \eqref{eq_conv_G_hat_X_times_Y}, the first part of this proof and Theorem 16.16 in~\cite{Kallenberg2002}, the convergence $\xi_n\big([0,t]\big) \overset{\mathcal{D}}{\longrightarrow } \xi\big([0,t]\big)$ holds true for each $ t >0$. Since $\xi_n$ and $\xi$ are point processes, $1/2$ is a point of continuity of the distribution functions of both $\xi_n\big([0,t]\big)$ and $\xi\big([0,t]\big)$, and we obtain
$$
\P\Big( \xi_n\big([0,t]\big) = 0\Big) = \P\left( \xi_n\big([0,t]\big) \le  \frac{1}{2}\right)
\to \P\left( \xi\big([0,t]\big) \le  \frac{1}{2}\right) = \P\Big( \xi\big([0,t]\big) = 0\Big)
$$
for each $ t >0$. Thus, we have
\begin{align*}
\P\big(t_1(\xi_n) \le t\big)
=& \; 1 - \P\big(t_1(\xi_n) > t\big) \\
=& \; 1 - \P\Big( \xi_n\big([0,t]\big) = 0\Big)\\
\to& \; 1 - \P\Big( \xi\big([0,t]\big) = 0\Big)\\
=& \; 1 - \P\big(t_1(\xi) > t\big) \\
=& \; \P\big(t_1(\xi) \le t\big).
\end{align*}
\qed
\end{proof}

\section{Generalizations 1 - Sets with unique diameter}
\label{sec_generalizations_unique_diameter}
This section deals with some obvious generalizations of \autoref{thm_main_result}. \autoref{subsec_more_general_densities_in_ellipsoids} is devoted to more general densities than those covered by \autoref{cond_density_p_l_p_r} in \autoref{sec_main_results}. Being more precise, we will investigate densities supported by ellipsoids that are allowed to tend to $0$ or $\infty $ close to the poles. It will turn out that the so-called Pearson Type II distributions are special distributions covered by this setting.
\autoref{subsec_joint_convergence_k_largest_distances} establishes a limit theorem for the joint convergence of the $k $ largest distances among the random points in the settings of both \autoref{sec_main_results} and \autoref{subsec_more_general_densities_in_ellipsoids}.
Moreover, \autoref{subsec_p_pllipsoids_p_norms} deals with $p$-superellipsoids and $p$-norms, where $1 \le p < \infty $. If the underlying $p$-superellipsoid has a unique diameter with respect to the $p$-norm and we use this norm to define the largest distance among the random points, we obtain very similar results as seen in \autoref{sec_main_results}.

\subsection{More general densities supported by ellipsoids}
\label{subsec_more_general_densities_in_ellipsoids}

In this section we consider closed ellipsoids
\begin{equation}
\label{eq_def_open_ellipsoid_Pearson}
E := \left\{z \in \R^d:  \sum_{k=1}^{d}\left( \frac{z_k}{a_k} \right)^2 \le 1\right\},
\end{equation}
with half axes $a_1 > a_2 \ge \ldots \ge a_d>0$, seen before in \autoref{cor_Ellipsoid}, and we define $\Sigma := \diag(a_1^2,\ldots,a_d^2) \in \R^{d\times d}$. Inside of these ellipsoids we  consider densities that satisfy the following condition:
\begin{condition}
\label{cond_f_gen_density_in_ellipsoid}
We assume $f: E\to \R_{+}$, $\int_{E}f(z)\,\mathrm{d}z =1 $ and that there are constants $\alpha_\ell,\alpha_r > 0$ and $\beta_\ell,\beta_r >-1$ so that the function
$$
z \mapsto \frac{f(z)}{\alpha_i\left(1- z^\top \Sigma^{-1} z\right)^{\beta_i}},
$$
that maps from $\text{int}(E)$ into $\R_{+}$,
can be extended continuously at the poles $(-a,\mathbf{0} )$ and $(a,\mathbf{0} )$ with value $1$. Thereby, $\alpha_\ell, \beta_\ell$ correspond to the left pole $(-a,\mathbf{0} )$ and $\alpha_r, \beta_r$ to the right pole $(-a,\mathbf{0} )$, respectively.
\end{condition}
\vspace{2mm}

Notice that \autoref{cond_density_p_l_p_r} was a special case of this condition, namely for $\beta_i = 0$ and with $\alpha_i = p_i$, $i \in \left\{ \ell,r\right\}$ (observe that we can use $E$ instead of $\text{int}(E)$ in this case).
The crucial difference to the setting of \autoref{thm_main_result} occurs in \autoref{lemma_transformation_density}. Before we state the main result of this section, which is \autoref{thm_main_result_gen_density_in_ellipsoid}, we will point out this essential difference.
As already seen in \autoref{cor_Ellipsoid}, we have
\begin{align*}
H_\ell = H_r = \diag\left( \frac{a_1}{a_2^2}\ ,\ \ldots\ ,\ \frac{a_1}{a_d^2}\right),
\end{align*}
and because of this symmetry, we briefly write $H := H_\ell = H_r$.
Remember now the construction of $P_1(H)$ given at the beginning of \autoref{subsec_proof_main_thm_geom_considerations}. In this section, we use the same construction for $\text{int}(E)$ instead of $E$ to avoid divisions by $0$ for $\beta<0$, and we conclude that
\begin{align*}
P_1(H)\
= &\ \left\{z \in \R^d:  \left( \frac{z_1-a_1}{a_1} \right)^2 + \sum_{k=2}^{d}\left( \frac{z_k}{a_k} \right)^2 < 1, z_1 < a_1\right\}
=  \ \left\{z \in \R^d:   \sum_{k=2}^{d}\left( \frac{z_k}{a_k} \right)^2 < \frac{2z_1}{a_1} - \left( \frac{z_1}{a_1}\right)^2, z_1 < a_1\right\}.
\end{align*}
Since we only consider distributions of $Z$ that are absolutely
continuous with respect to Lebesgue measure, this is no restriction at all.
To show an adjusted version of \autoref{lemma_transformation_density}, we have, in generalization of~\eqref{eq_def_of_nu}, to define the constant
$$
\nu := \frac{1}{d+1+2\beta},
$$
$\beta>-1$, the rescaling function
$$
T_n(z) := \left( \ n^{2\nu}z_1\ ,\ n ^{\nu}\widetilde z\ \right)
$$
for $n \in \N$, $ z = (z_1,\widetilde z) \in \R^d$ and the (now open) limiting set
$$
 P(H)  := \ \left\{ z \in \R^d:  \sum_{k=2}^{d}\left( \frac{z_k}{a_k}\right)^2 <  \frac{2z_1}{a_1}\right\}.
$$
Now we can state an adapted version of \autoref{lemma_transformation_density}.
\begin{lemma}
\label{lemma_transformation_density_gen_density}
Suppose the random vector $\ZVhelp = (\ZVhelp_1,\ldots,\ZVhelp_d)$ has a density $g$ on $P_1(H)$ satisfying
\begin{align*}
g(z)
& =
%\alpha\left(1 - \big( z - (a_1,\mathbf{0} )\big)^\top  \Sigma^{-1} \big( z - (a_1,\mathbf{0})\big)\right)^{\beta}
%=
\left( 1+ o(1)\right)\cdot \alpha \cdot \left( 1 - \left( \frac{z_1-a_1}{a_1}\right)^2 - \sum_{k=2}^{d}\left( \frac{z_k}{a_k}\right)^2 \right)^{\beta},
\label{eq_Lemma_Trans_gen_density_Condition}
\end{align*}
uniformly on $P_1(H) \cap \left\{ z_1 \le \delta\right\}$ as $\delta \to 0$, for some $\alpha>0$ and $ \beta > -1$.
Then, for every bounded Borel set $\boundSet \subset \R^d$, we have
\begin{equation*}
\label{eq_gen_density_P_T_n_V_in_B}
\P\big( T_n(\ZVhelp) \in \boundSet \big)= \frac{\alpha}{n}\cdot\kappa_n(\boundSet)
\end{equation*}
with $\kappa_n(\boundSet) \to \Lambda_{\beta}(\boundSet)$ and
\begin{equation*}\label{eq_Def_Lambda_beta}
\Lambda_{\beta}(\boundSet)
:=  \int_{\boundSet} \left( \frac{2z_1}{a_1} -  \sum_{k=2}^{d}\left( \frac{z_k}{a_k}\right)^2 \right)^{\beta} \ind\left\{ z \in P(H)\right\}\,\mathrm{d}z.
\end{equation*}
\end{lemma}
The proof of this lemma is very technical since we can (in general) neither apply the dominated convergence theorem, nor the monotone convergence theorem to show $\kappa_n(\boundSet) \to \Lambda_{\beta}(\boundSet)$. Instead, an application of Scheff\'e's Lemma is necessary, see \cite{Schrempp2017} for more details and for the connection between this lemma and the following result.

\begin{theorem}
\label{thm_main_result_gen_density_in_ellipsoid}
Let the density $f$ be supported by the ellipsoid $E$ with half-axes $a_1 > a_2 \ge \ldots \ge a_d>0$ and satisfy \autoref{cond_f_gen_density_in_ellipsoid} with $\beta_\ell = \beta_r =:\beta$. We then have
\begin{equation*}
%\label{eq_theorem_main_result_gen_density_in_ellipsoid}
n^{\frac{2}{d+1+2\beta}}\big(2a_1 - \mathrm{diam}(\mathbf{Z}_n )\big) \overset{\mathcal{D}}{\longrightarrow } \min_{i,j \ge 1} \left\{ \mathcal{X}_{i,1}  + \mathcal{Y}_{j,1} - \frac{1}{4a_1}\big| \widetilde {\mathcal{X}}_{i}-\widetilde {\mathcal{Y}}_{j}\big|^2 \right\},
\end{equation*}
where $\left\{ \mathcal{X}_i, i \ge 1\right\} \overset{\mathcal{D}}{= }  \PRM\big(  \alpha_\ell\cdot \Lambda_{\beta}\big) $ and $\left\{ \mathcal{Y}_j, j \ge 1\right\} \overset{\mathcal{D}}{= }  \PRM\big(\alpha_r\cdot \Lambda_{\beta}\big) $ are independent Poisson processes. If \autoref{cond_f_gen_density_in_ellipsoid}  and  -- without loss of generality -- the inequality $\beta_\ell > \beta_r$ hold true, we obtain
\begin{equation*}
%\label{eq_theorem_main_result_gen_density_in_ellipsoid_diff_beta}
n^{\frac{2}{d+1+2\beta_{\ell}}}\big(2a_1 - \mathrm{diam}(\mathbf{Z}_n )\big) \overset{\mathcal{D}}{\longrightarrow } \min_{i \ge 1} \left\{ \mathcal{X}_{i,1} - \frac{1}{4a_1}\big| \widetilde  {\mathcal{X}}_{i}\big|^2 \right\},
\end{equation*}
with $\left\{ \mathcal{X}_i, i \ge 1\right\} \overset{\mathcal{D}}{= }  \PRM\big(  \alpha_\ell\cdot \Lambda_{\beta_\ell}\big) $.
The same results hold true if we replace $ \mathrm{diam}(\mathbf{Z}_n )$ with $M_n$.
\end{theorem}
\vspace{2mm}
\begin{proof}
Under \autoref{cond_f_gen_density_in_ellipsoid} we have $f(z)>0$ for each $z$ arbitrarily close to one of the poles. In the case $\beta_\ell = \beta_r$, this inequality allows us to copy the proof of \autoref{thm_main_result} almost completely. The only difference is that we have to apply \autoref{lemma_transformation_density_gen_density} instead of \autoref{lemma_transformation_density} to show an adapted version of \autoref{lemma_Conv_of_V_n}. In the case $\beta_\ell>\beta_r$ we will observe a higher magnitude of points lying close to the right pole than to the left.
This higher magnitude has far-reaching implications for the proof to follow.
First of all, we define
$$
\nu_\ell:=\frac{1}{d+1+2\beta_\ell}, \qquad
T_n^{\ell}(z) := \left( \ n^{2\nu_\ell}z_1\ ,\ n ^{\nu_\ell}\widetilde z\  \right)
\qquad \text{and} \qquad
P_n^{\ell}(H) := T_n^{\ell}\big(P_1(H)\big).
$$
The beginning of the main part of the proof of \autoref{thm_main_result} in \autoref{subsubsec_main_part_proof_main_thm} can be copied in this case, too. We will only point out the main difference.
Let $N_{r,n}$ and $ Y_1,Y_2,\ldots $ be defined as in the proof of \autoref{thm_main_result} and write $\ZVhelp^r := (a-Y_{1,1},\widetilde Y_1)$. Then,
$$
\widetilde{\mathbf{Y} }_n := \sum_{j=1}^{N_{r,n}}\varepsilon_{(a-Y_{j,1},\widetilde Y_{j})}
$$
is a Poisson process with intensity measure $nc_{r,\delta}\cdot \P_{\ZVhelp^r}$, and
$\mathbf{Y}_n^{\ell} := \mathbf{\widetilde Y}_n \circ (T_n^{\ell})^{-1}$ -- taking the part of $\mathbf{Y}_n $ in the proof of \autoref{thm_main_result} --  is a Poisson process with intensity measure $\widehat \mu_n := nc_{r,\delta} \cdot \P_{\ZVhelp^r} \circ (T_n^{\ell})^{-1}$. The density $f$ fulfills \autoref{cond_f_gen_density_in_ellipsoid} at the right pole with power $\beta_r$, but the shifted process $\widetilde{\mathbf{Y} }_n $ is scaled via $T_n^\ell$, which depends on $\beta_\ell$, \emph{not} on $ \beta_r$. Broadly speaking, this `wrong' (too slow) scaling has the effect, that $\mathbf{Y}_n^{\ell} $ will generate more and more points arbitrarily close to $\mathbf{0} $, see \cite{Schrempp2017} for technical details.
\qed
\end{proof}

\begin{example}
\label{ex_pearson_type_II}
We now consider the so-called $d$-dimensional symmetric multivariate Pearson Type II distributions supported by an ellipsoid with half-axes $a_1 > a_2 \ge \ldots \ge a_d>0$, where $d \ge 2$. According to equation $($2.43$)$ in~\cite{Fang1990} and Example 2.11 in the same reference, we know that the corresponding densities are given by
\begin{align*}
  f_{\beta}(z)
  %&= \frac{1}{\sqrt{\prod_{i=1}^{d}a_i^2}} \cdot \frac{\Gamma\left( \frac{d}{2}+\beta + 1\right)}{\Gamma\left( \beta +1\right)\pi^{\frac{d}{2}}}(1-z^\top \Sigma^{-1}z)^{\beta}\cdot \ind\left\{ z \in E\right\}\\
&=\frac{\Gamma\left( \frac{d}{2}+\beta + 1\right)}{\Gamma\left( \beta +1\right)\pi^{\frac{d}{2}}\prod_{i=1}^{d}a_i}\left(1-z^\top \Sigma^{-1}z\right)^{\beta}\cdot \ind\left\{ z \in \text{int}(E)\right\}.
\end{align*}
Hence, \autoref{cond_f_gen_density_in_ellipsoid} holds true with $\beta_\ell = \beta_r = \beta$ and
$
\alpha:=\alpha_\ell = \alpha_r = \frac{\Gamma\left( \frac{d}{2}+\beta + 1\right)}{\Gamma\left( \beta +1\right)\pi^{\frac{d}{2}}\prod_{i=1}^{d}a_i},
$
so that we can apply \autoref{thm_main_result_gen_density_in_ellipsoid}.
\end{example}

Figures~\ref{fig_pearson_beta_minus_0_5} and \ref{fig_pearson_beta_2} illustrate the densities $f_{\beta}$ and the corresponding densities of the intensity measures $\alpha \cdot \Lambda_\beta$ in the setting of \autoref{ex_pearson_type_II} for $d=2$, $a_1 = 1, a_2 = 1/2$ and $\beta \in \left\{ -1/2, 2\right\}$.
See \cite{Schrempp2017} for the illustration of some more Pearson Type II densities in two dimensions and the results of a simulation study.

\begin{figure}[!ht]
\vspace{5mm}
\begin{center}
\begin{minipage}{0.4\textwidth}
\hspace{-13mm}
\includegraphics[width=0.6\textwidth, angle=-90,trim = 60 0 100 0 ] {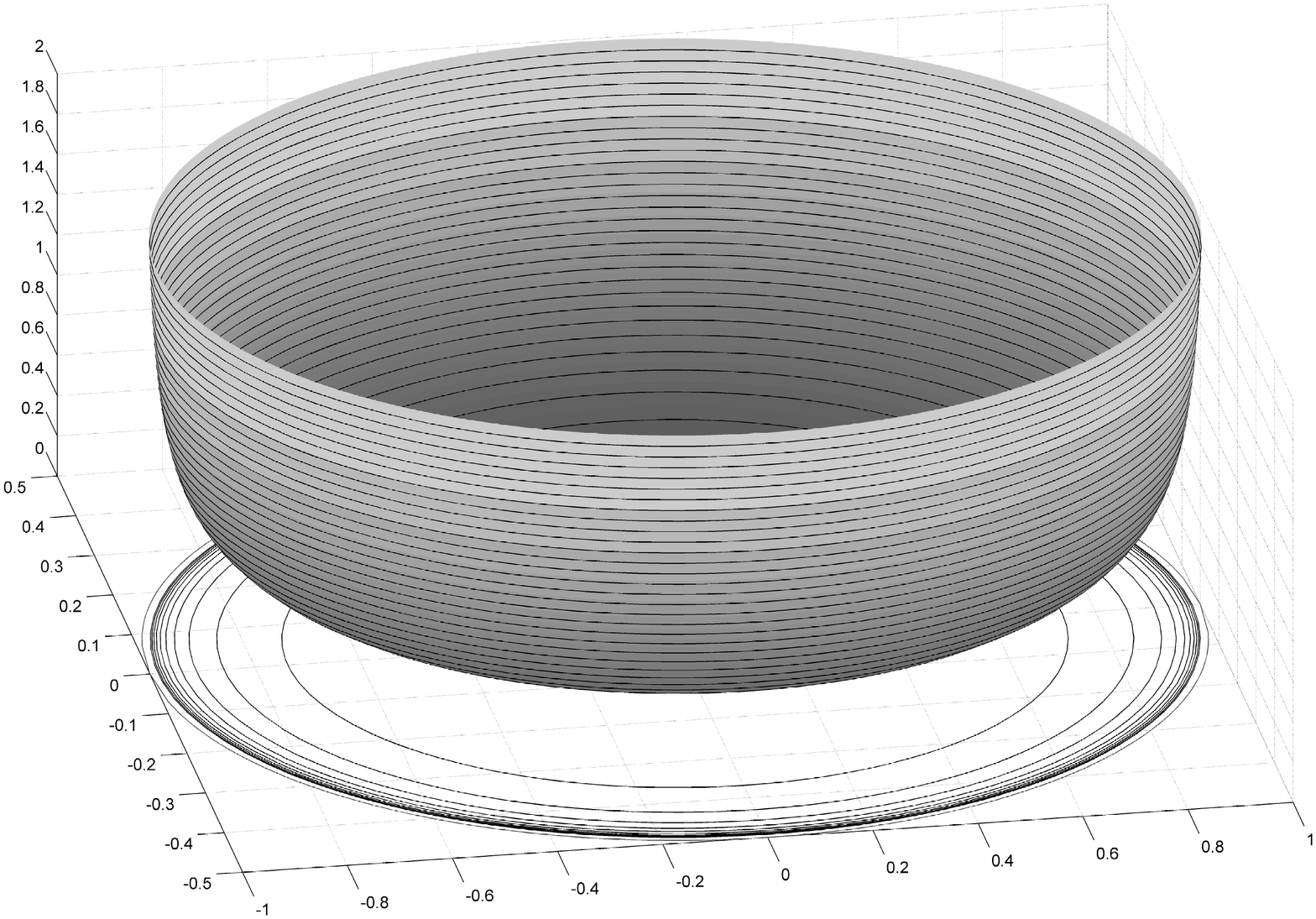} \\
\end{minipage} \hspace{4mm}
\begin{minipage}{0.4\textwidth}
\hspace{-5mm}
\includegraphics[width=0.6\textwidth, angle=-90,trim = 60 0 100 0 ] {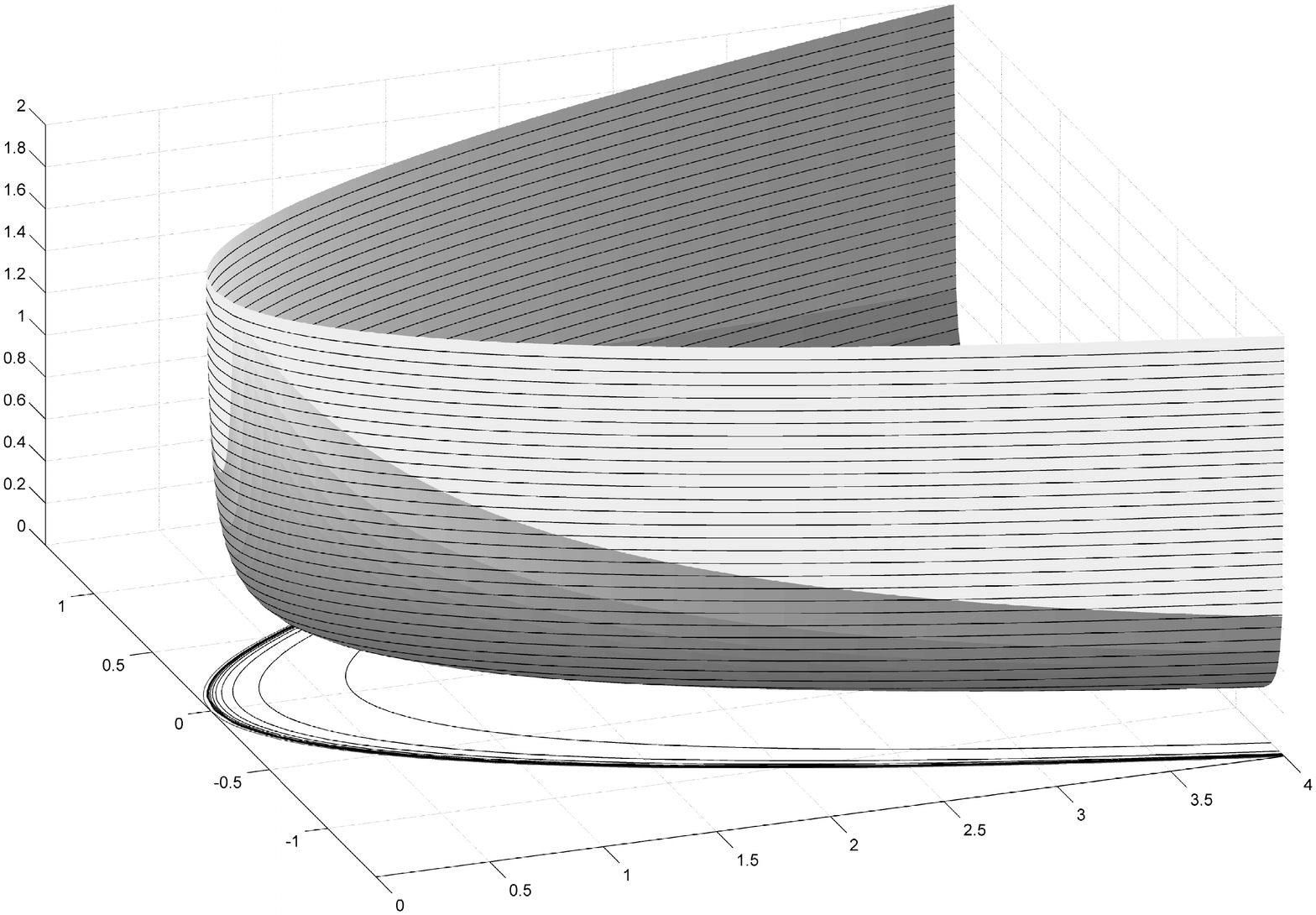}\\
\begin{center}
\vspace{-3mm}
\end{center}
\end{minipage}\\[2mm]
\end{center}
\caption{The density $f_{\beta}$ (left) and that of the intensity measure $\alpha \cdot \Lambda_\beta$ (right) in the setting of \autoref{ex_pearson_type_II} for $d=2$ with $a_1 =1, a_2 = 1/2$ and $\beta = -1/2$.}
\label{fig_pearson_beta_minus_0_5}
\end{figure}

\begin{figure}[!ht]
\vspace{5mm}
\begin{center}
\begin{minipage}{0.4\textwidth}
\hspace{-13mm}
\includegraphics[width=0.6\textwidth, angle=-90,trim = 60 0 100 0 ] {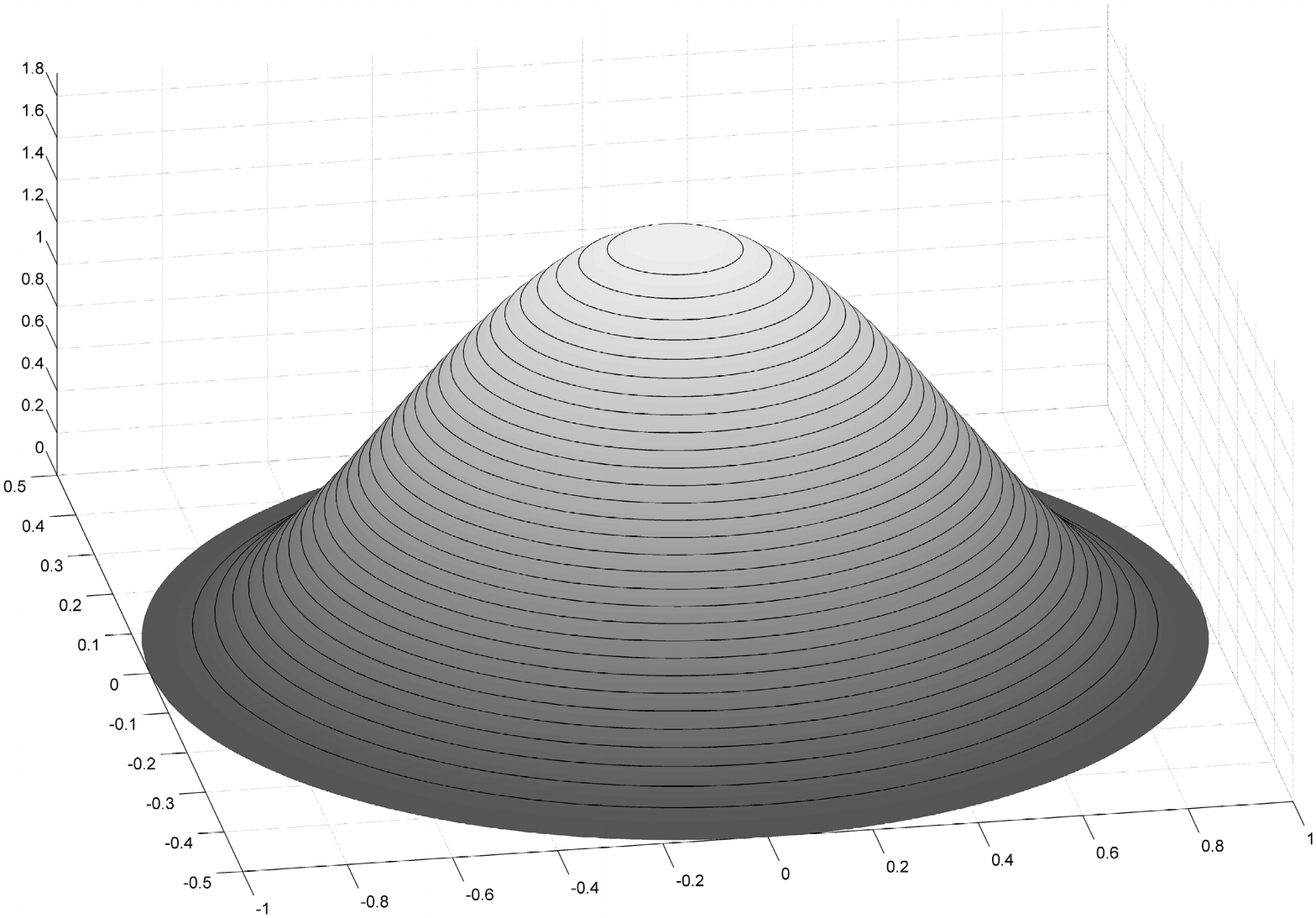} \\
\begin{center}
\vspace{-3mm}
\end{center}
\end{minipage} \hspace{4mm}
\begin{minipage}{0.4\textwidth}
\hspace{-5mm}
\includegraphics[width=0.6\textwidth, angle=-90,trim = 60 0 100 0 ] {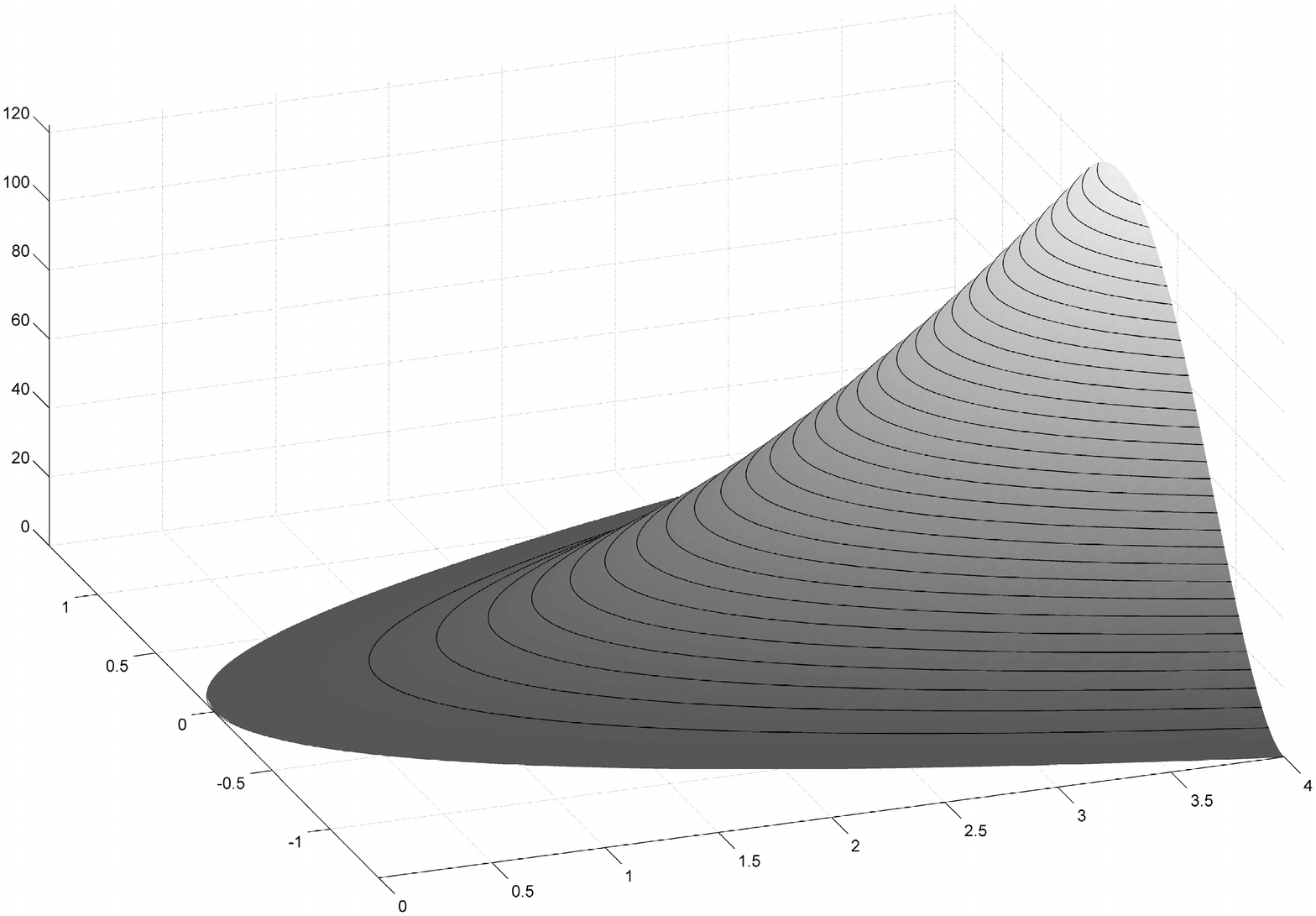}\\
\begin{center}
\vspace{-3mm}
\end{center}
\end{minipage}\\[3mm]
\end{center}
\caption{The density $f_{\beta}$ (left) and that of the intensity measure $\alpha \cdot \Lambda_\beta$ (right) in the setting of \autoref{ex_pearson_type_II} for $d=2$ with $a_1 =1, a_2 = 1/2$ and $\beta = 2$.}
\label{fig_pearson_beta_2}
\end{figure}

\subsection{Joint convergence of the \texorpdfstring{$k$}{k} largest distances}
\label{subsec_joint_convergence_k_largest_distances}
To state a result on the joint asymptotical behavior of the $k$ largest distances of the Poisson process $\mathbf{Z}_n = \sum_{i=1}^{N_n}\varepsilon_{Z_{i}}$, introduced in \autoref{sec_fundamentals}, we need some additional definitions.
For $n \in \N$, let $D_n^{(1)} \ge D_n^{(2)} \ge \ldots \ge D_n^{\left( k\right)}$ be the $k$ largest distances in descending order between $Z_i$ and $Z_j$ for $1 \le i< j \le N_n$. So, we especially have $D_n^{(1)} =   \diam(\mathbf{Z}_n ).$
For a point process $\xi$ on $\R_+$ and $i \in \N$ we define $t_i(\xi) := \inf\left\{ t: \xi\big([0,t]\big) \ge i\right\}.$
According to Proposition 9.1.XII in~\cite{Daley2008}, each $t_i(\xi)$ is a well-defined random variable if $\xi$ is a simple point process. Since the point processes $ \widehat G(\mathbf{X}_n \times \mathbf{Y}_n)$ and $ \widehat G(\mathbf{X} \times \mathbf{Y})$ on $\R_+$ (introduced in the proof of \autoref{thm_main_result}) are simple, we conclude that the random variables $t_i\big( \widehat G(\mathbf{X}_n \times \mathbf{Y}_n) \big)$ and $t_i\big( \widehat G(\mathbf{X} \times \mathbf{Y}) \big)$ are well-defined for each fixed $ i \in \N$.
%See \eqref{eq_conv_of_X_and_Y} and \eqref{eq_def_G_hat} in the proof of \autoref{thm_main_result} for the definitions of the aforementioned point processes and of the function $\widehat G$.
Now we can state our result on the joint convergence of the $k$ largest distances in the setting of \autoref{sec_main_results}:
\vspace{2mm}
\begin{theorem}
\label{thm_main_result_joint_conv}
If Conditions~\ref{cond_unique_diameter} to \ref{cond_density_p_l_p_r} hold true, we have for each $k \in \N$ the joint convergence
\begin{equation*}
n^{\frac{2}{d+1}}
\Big(
  2a - D_n^{(1)}\ ,\ldots, \ 2a - D_n^{(k)}
\Big)
 \overset{\mathcal{D}}{\longrightarrow }
\Big(
  t_1\big( \widehat G(\mathbf{X} \times \mathbf{Y}) \big)\ ,\ldots, \ t_k\big( \widehat G(\mathbf{X} \times \mathbf{Y}) \big)
\Big)
 ,
\end{equation*}
where $\mathbf{X}  \overset{\mathcal{D}}{= }  \PRM\big(p_\ell\cdot m_d\big|_{P(H_\ell)}\big) $ and $\mathbf{Y}  \overset{\mathcal{D}}{= }  \PRM\big(p_r\cdot m_d\big|_{P(H_r)}\big) $ are independent Poisson processes.
\end{theorem}
The proof of this theorem is a simple generalization of that of \autoref{thm_main_result}, see Section 5.3 in \cite{Schrempp2017} for more details.
We can immediately generalize the results of \autoref{subsec_more_general_densities_in_ellipsoids}, too:
\vspace{2mm}
\begin{theorem}
\label{thm_joint_convergence_gen_density}
Let the density $f$ be supported by the ellipsoid $E$ defined in~\eqref{eq_def_open_ellipsoid_Pearson} with half-axes $a_1 > a_2 \ge \ldots \ge a_d>0$, and put $a := a_1 $. If $f$ satisfies \autoref{cond_f_gen_density_in_ellipsoid} with $\beta_\ell = \beta_r =:\beta$ then, for each fixed $k \ge 1$, we have
\begin{equation*}
n^{\frac{2}{d+1+2\beta}}
\Big(
  2a - D_n^{(1)}\ ,\ \ldots\ ,\  2a - D_n^{(k)}
\Big)
 \overset{\mathcal{D}}{\longrightarrow }
\Big(
  t_1\big( \widehat G(\mathbf{X} \times \mathbf{Y}) \big)\ ,\ \ldots\ , \ t_k\big( \widehat G(\mathbf{X} \times \mathbf{Y}) \big)
\Big),
\end{equation*}
where $\mathbf{X} \overset{\mathcal{D}}{= }  \PRM\big(  \alpha_\ell\cdot \Lambda_{\beta}\big) $ and $\mathbf{Y} \overset{\mathcal{D}}{= }  \PRM\big(\alpha_r\cdot \Lambda_{\beta}\big) $ are independent Poisson processes.
\end{theorem}
\vspace{2mm}

Notice that the definition $a := a_1$ in the theorem above is necessary, since the function $\widehat G$ has been defined in \autoref{subsubsec_main_part_proof_main_thm} in terms of $a$, not of $a_1$.
See \cite{Schrempp2017} for the results of a simulation study.

\subsection{\texorpdfstring{$p$}{p}-superellipsoids and \texorpdfstring{$p$}{p}-norms}
\label{subsec_p_pllipsoids_p_norms}

For $1 \le p < \infty$  and $a_1 > a_2 \ge a_3 \ge \ldots \ge a_d >0$ we define the so-called $p$-superellipsoid
$$
E^p := \left\{ z \in \R^d: \sum_{k=1}^{d}\left( \frac{|z_k|}{a_k}\right)^p \le 1\right\}
$$
and the corresponding $p$-norm
$$
|z|_p := \left( \sum_{k=1}^{d}|z_k|^p \right)^{\frac{1}{p}}, \quad z \in \R^d.
$$
Moreover, based on this norm, let
$$
\diam_p(A) := \sup_{x,y \in A} |x-y|_p
$$
be the so-called $p$-diameter of a set $A \subset \R^d$.
The definitions of $E^p$ and $|\cdot|_p$ yield
$\big|(-a_1,\mathbf{0} )-(a_1,\mathbf{0} )\big|_p = 2a_1$, and in view of  $a_1 > a_2 \ge a_3 \ge \ldots \ge a_d >0$ we have $|z|_p \le a_1$ for each $z \in E^p$, with equality only for $z \in \big\{ (-a_1,\mathbf{0} ),(a_1,\mathbf{0} )\big\}$. Together with $|x-y|_p \le |x|_p + |y|_p$ for all $x,y \in \R^d$ we can infer that the set $E^p$ has a unique diameter of length $2a_1$ with respect to the $p$-norm between the points $(-a_1,\mathbf{0} )$ and $(a_1,\mathbf{0} )$.\\

We assume that the random variables $Z_1,Z_2,\ldots$ are i.i.d. with a common density $f$, supported by the superellipsoid $E^p$. As in \autoref{sec_main_results}, we consider densities that are continuous and bounded away from 0 at the poles.  In this subsection we will investigate the largest distance between these random points with respect to the corresponding $p$-norm, not with respect to the Euclidean norm, i.e. we consider
$$
M_n^{p} := \max_{1 \le i , j \le n}|Z_i - Z_j|_p.
$$
Using the Poisson process $\mathbf{Z}_n$ with intensity measure $n\P_{Z}$, defined in \autoref{sec_fundamentals}, we get
$$
\diam_p(\mathbf{Z}_n ) =   \max_{1 \le i,j \le N_n} \big|Z_i - Z_j\big|_p.
$$
Defining the new limiting set
\begin{equation*}
\label{eq_p_ellips_def_limiting_set}
 P^p := \left\{ z \in \R^d: \sum_{k=2}^d \left( \frac{|z_k|}{a_k}\right)^p \le \frac{pz_1}{a_1}\right\}
\end{equation*}
and using very similar techniques as seen before in \autoref{sec_proof_main_thm},
we can prove the following result:
\vspace{2mm}
\begin{theorem}
\label{thm_p_ellipsoid_p_norm}
Under the standing assumptions of this section and if \autoref{cond_density_p_l_p_r} holds true for $E$ replaced with $E^p$ and $a = a_1$, then
\begin{equation*}
n^{\frac{p}{d+p-1}}\big(2a_1 - \mathrm{diam}_p(\mathbf{Z}_n )\big) \overset{\mathcal{D}}{\longrightarrow } \min_{i,j \ge 1} \left\{ \mathcal{X}_{i,1}  + \mathcal{Y}_{j,1} - \frac{1}{p(2a_1)^{p-1}} \big|\widetilde {\mathcal{X}}_{i}- \widetilde {\mathcal{Y}}_{j}|_p^p \right\},
\end{equation*}
where $\left\{ \mathcal{X}_i, i \ge 1\right\} \overset{\mathcal{D}}{= }  \PRM\big(p_\ell\cdot m_d\big|_{P^p}\big) $ and $\left\{ \mathcal{Y}_j, j \ge 1\right\} \overset{\mathcal{D}}{= }  \PRM\big(p_r\cdot m_d\big|_{P^p}\big) $ are independent Poisson processes. The same holds true if we replace $ \diam_p(\mathbf{Z}_n )$ with $M_n^{p}$.
\end{theorem}
\vspace{2mm}
The proof of this theorem can be found in Section 5.5 in \cite{Schrempp2017}.

\begin{corollary}
\label{cor_p_norm}
Given the uniform distribution on $E^p$, \autoref{cond_density_p_l_p_r} holds true for $E$ replaced with $E^p$, $a=a_1$ and
$$
p_\ell = p_r = \frac{1}{m_d\big(E^p \big)} = \left( \frac{\left( 2 \Gamma\left( 1 + \frac{1}{p}\right) \right)^d\prod_{i=1}^d a_i}{\Gamma\left( 1 + \frac{d}{p}\right)} \right)^{-1} > 0,
$$
see~\cite{Wang2005}. We can thus apply \autoref{thm_p_ellipsoid_p_norm}.
%For $d=2$, $a_1 = 1, a_2 = 1/2$ and $p \in \left\{1,3/2,4\right\}$, the sets $E^p$ and $P^p$ and the results of a simulation study are illustrated in the Figures~\ref{fig_p_norm_set_p_1} to \ref{fig_p_norm_ecdf_p_4}.
Notice that \autoref{cor_ellipse_uniform} is a special case of this corollary, namely for $p = 2$.
\end{corollary}

Some more generalizations can be found in \cite{Schrempp2017}:
Section 5.2 in that reference takes a look at more general densities supported by any set (not only ellipsoids), fulfilling the Conditions~\ref{cond_unique_diameter} to \ref{cond_A_eta_pos_semi_definite}. Furthermore, a different shape of $E$ close to the poles is considered in Section 5.4, and Section 5.6  illustrates that the smoothness of the boundary of $ E$ at the poles, as demanded in  \autoref{cond_shape_pole_caps}, is by no means necessary to prove results similar to that of \autoref{thm_main_result}.

\section{Generalizations 2 - Sets with no unique diameter}
\label{sec_generalizations_no_unique_diameter}

In this section we consider sets with no unique diameter, i.e. we no longer assume that \autoref{cond_unique_diameter} holds true. Basically, there are two different ways to modify this condition. The first is given by sets, having $ k  $ pairs of poles, where $1 < k < \infty $, see \autoref{cond_several_axes_k_pole_pairs} below for a formal definition. Such sets will be studied in \autoref{subsec_several_major_axes}. An alternative modification of \autoref{cond_unique_diameter} is  -- heuristically spoken in three dimensions -- given by sets with an equator, for example a three-dimensional ellipsoid with half-axes $1,1$ and $1/2$. For Pearson Type II distributed points in $d$-dimensional ellipsoids with at least two but less than $d$ major half-axes,  we still do not know whether a limit distribution for $M_n$ exists, or not. However, at least for each of these Pearson Type II distributions, \autoref{subsec_ellipsoid_several_major_axes} exhibits bounds for the limit distribution of $M_n$, provided that such a limit law exists.

\subsection{Several major axes}
\label{subsec_several_major_axes}
In this subsection we consider closed sets with more than one, but finitely many pairs of poles. To this end, we formulate a more general version of \autoref{cond_unique_diameter}:

\vspace{2mm}
\begin{condition}
\label{cond_several_axes_k_pole_pairs}
Let $E \subset \R^d$ be closed, $a > 0 $ , $k \ge 2 $ and $x^{(1)},\ldots,x^{(k)},y^{(1)},\ldots, y^{(k)} \in E $ so that $$
\diam(E) = \big|x^{(1)}-y^{(1)}\big| = \ldots = \big|x^{(k)}-y^{(k)}\big| = 2a
$$
and
\begin{equation}
\label{eq_several_axes_pairs_of_poles_different}
\left( x^{(i)},y^{(i)}\right) \ne \left( x^{(j)},y^{(j)}\right) \ne \left( y^{(i)},x^{(i)}\right)
\end{equation}
for $i \ne j$.
Furthermore, we assume
\begin{equation*}
%\label{eq_cond_unique_diameter_several_poles}
|x-y| < 2a \qquad \text{for each} \qquad (x,y)\in \big( E\backslash \left\{ x^{(1)},\ldots,x^{(k)},y^{(1)},\ldots, y^{(k)} \right\} \big) \times E.
\end{equation*}
\end{condition}

Observe that \eqref{eq_several_axes_pairs_of_poles_different} makes sure that no pair of poles (points with distance $2a$) is considered twice.
We want to emphasize the assumption $k < \infty $ in \autoref{cond_several_axes_k_pole_pairs}. Sets with an equator -- like an ellipsoid in $\R^{3}$ with half-axes $a_1 = a_2 > a_3$ -- are explicitly excluded by this condition, see \autoref{subsec_ellipsoid_several_major_axes} for some considerations in this setting.\\

For $m \in \left\{ 1,\ldots,k\right\}$, let $\phi^{(m)}$ be a rigid motion of $\R^d$ with $\phi^{(m)}\big( x^{(m)}\big) = (-a,\mathbf{0} )$ and $\phi^{(m)}\big( y^{(m)}\big) = (a,\mathbf{0} )$. If $f$ is a density with support $E$, we write $f^{(m)} := f \circ (\phi^{(m)})^{-1}$ for the transformed density supported by $\phi^{(m)}(E)$.
Our basic assumption in this section will be that, for each $m \in \left\{ 1,\ldots,k\right\}$, the set $\phi^{(m)}(E)$ and the density $f^{(m)}$ fulfill all the requirements of \autoref{thm_main_result}, formally:
\vspace{2mm}
\begin{condition}
\label{cond_several_axes_polecaps}
For each $m \in \left\{ 1,\ldots,k\right\}$, we assume that $\phi^{(m)}(E)$ satisfies Conditions~\ref{cond_shape_pole_caps} and \ref{cond_A_eta_pos_semi_definite}, and that the density $f^{(m)}$ fulfills \autoref{cond_density_p_l_p_r} with respect to some constants $p_\ell^{(m)},p_r^{(m)}>0$.
\end{condition}
\vspace{2mm}

\cite{Appel2002} investigated a similar setting in two dimensions for sets with boundary functions that -- in contrast to \autoref{cond_several_axes_polecaps} -- decay faster to zero at the poles than a square-root. In that setting, it was necessary to demand that any two different major axes have no vertex in common. Under \autoref{cond_several_axes_polecaps}, this requirement is given by definition: None of the points $x^{(1)},\ldots,x^{(k)},y^{(1)},\ldots,y^{(k)}$ can be part of more than one pair of points with distance $2a$, or, in other words, the set $E$ has exactly $2k$ poles, see Lemma 6.1 in \cite{Schrempp2017} for some more details.
Writing $B_{\varepsilon}(z)$ for the closed ball with center $z \in \R^d$, we can infer that there exists an $\varepsilon >0$ so that the balls
$
B_{\varepsilon}\big( x^{(1)}\big),\ldots,B_{\varepsilon}\big( x^{(k)}\big),B_{\varepsilon}\big( y^{(1)}\big),\ldots,B_{\varepsilon}\big( y^{(k)}\big)
$
are pairwise disjoint. For $m \in \left\{ 1,\ldots,k\right\}$ we define the set
$$
E^{(m)} := E \cap \left( B_{\varepsilon}\big( x^{(m)}\big) \cup B_{\varepsilon}\big( y^{(m)}\big)\right).
$$
After moving $E^{(m)}$ via $\phi^{(m)}$ into the suitable position, \autoref{thm_main_result} is applicable for each $m \in \left\{ 1,\ldots,k\right\}$. We consider again the Poisson process $\mathbf{Z}_n= \sum_{i=1}^{N_n}\varepsilon_{Z_{i}} $, defined in \autoref{sec_fundamentals}. Since the sets $E^{(1)},\ldots,E^{(k)}$ are pairwise disjoint, the restrictions $\mathbf{Z}_n\big( \cdot \cap E^{(1)}\big),\ldots, \mathbf{Z}_n\big( \cdot \cap E^{(k)}\big)$ are independent Poisson processes.
Consequently, for $m \in \left\{ 1,\ldots,k\right\}$, the maximum distances of points lying in $E^{(m)}$ are independent random variables. With
\begin{align*}
I_{n}^{(m)} &:= \big\{ (i,j) : 1 \le  i, j \le N_n, (Z_i,Z_j) \in E^{(m)} \times E^{(m)} \big\}
\end{align*}
for $m \in \left\{ 1,\ldots,k\right\}$, we obtain $I_n^{(m)} \ne \emptyset$ for sufficiently large $n$ for each $ m \in \left\{ 1,\ldots,k\right\}$ almost surely and hence
\begin{align*}
2a - \max_{1\le i,j \le N_n}|Z_i - Z_j|
= 2a - \max_{1 \le m \le k} \left\{ \max_{(i,j) \in I_n^{(m)}}|Z_i - Z_j|\right\}
%= 2a + \min_{1 \le m \le k} \left\{ -\max_{(i,j) \in I_n^{(m)}}|Z_i - Z_j|\right\}
&= \min_{1 \le m \le k} \left\{2a -\max_{(i,j) \in I_n^{(m)}}|Z_i - Z_j|\right\}.
%&= \min_{1 \le m \le k} \left\{2a + \min_{(i,j) \in I_n^{(m)}}-|Z_i - Z_j|\right\}\\
%&= \min_{1 \le m \le k} \left\{ \min_{(i,j) \in I_n^{(m)}}\big(2a-|Z_i - Z_j|\big)\right\}\\
%&= \min_{1 \le m \le k} \left\{ 2a-\max_{(i,j) \in I_n^{(m)}}\big(|Z_i - Z_j|\big)\right\}.
\end{align*}

As mentioned before, we can apply \autoref{thm_main_result} to each of the random variables $\max_{(i,j) \in I_n^{(m)}}|Z_i - Z_j|$,
and since these $k$ random variables are independent for each $n \in \N$, the $k$ limiting random variables inherit this property. Hence, we obtain as limiting distribution of the maximum distance of points within $E$ a minimum of $k$ independent random variables, each of which can be described as seen in \autoref{thm_main_result}. After stating one last definition we can formulate a generalized version of our main result \autoref{thm_main_result}.
Instead of $H_\ell$ and $H_r$ we write $H_\ell^{(m)}$ and $H_r^{(m)}$ for the Hessian matrices of the corresponding boundary functions of $E^{(m)}$ at the poles, $m \in \{1,\ldots,k\}$.
\vspace{2mm}
\begin{theorem}
\label{thm_several_major_axes}
Under \autoref{cond_several_axes_k_pole_pairs} and \autoref{cond_several_axes_polecaps} we have
$$
n^{\frac{2}{d+1}}\big(2a - \mathrm{diam}(\mathbf{Z}_n )\big) \overset{\mathcal{D}}{\longrightarrow }
\min_{1 \le m \le k} Z^{(m)},
$$
with independent random variables $Z^{(1)},\ldots,Z^{(k)}$, fulfilling
\begin{equation*}
Z^{(m)}\overset{\mathcal{D}}{= }  \min_{i,j \ge 1} \left\{ \mathcal{X}_{i,1}^{(m)}  + \mathcal{Y}_{j,1}^{(m)} - \frac{1}{4a} \big|\widetilde {\mathcal{X}}_{i}^{(m)}- \widetilde {\mathcal{Y}}_{j}^{(m)}\big|^2\right\},
\end{equation*}
where all the Poisson processes $\left\{ \mathcal{X}_i^{(m)}, i \ge 1\right\} \overset{\mathcal{D}}{= }  \PRM\Big(p_\ell^{(m)}\cdot m_d\big|_{P\big(H_\ell^{(m)}\big)}\Big) $, $\left\{ \mathcal{Y}_j^{(m)}, j \ge 1\right\} \overset{\mathcal{D}}{= }  \PRM\Big(p_r^{(m)}\cdot m_d\big|_{P\big(H_r^{(m)}\big)}\Big) $, $m \in \left\{ 1,\ldots,k\right\}$, are independent. The same result holds true if we replace $ \mathrm{diam}(\mathbf{Z}_n )$ with $M_n$.
\end{theorem}
\vspace{2mm}

See \cite{Schrempp2017} for an application of this theorem to the ball of radius
$r > 0$ with respect to the $p$-norm for $p > 2$.

\subsection[Ellipsoids with no unique major half-axis]{Ellipsoids with no unique major half-axis }

\label{subsec_ellipsoid_several_major_axes}
In this subsection, we fix $d \ge 3$ and $e \in \left\{ 2,\ldots,d-1\right\}$ and consider the  $d$-dimensional ellipsoid $E$ with half-axes $a_1 = \ldots =a_e = 1$ and $1 > a_{e+1} \ge \ldots \ge a_{d}$, formally:
$$
E = \left\{ z \in \R^d: z_1^2 + \ldots +z_e^2 + \left( \frac{z_{e+1}}{a_{e+1}}\right)^2 + \ldots + \left( \frac{z_d}{a_d}\right)^2 \le 1\right\}.
$$
There is no loss of generality in assuming that the $e$ major half-axes have length $1$. Otherwise, one would only have to scale $E$ and $M_n$ in a suitable way. We assume that the points $Z_1,Z_2,\ldots$ are independent and identically distributed according to a Pearson Type II distribution with parameter $\beta > -1$ on $\text{int}(E)$. This means that the density of $Z_1$ is given by
$$
f(z) = c_1 \cdot \left(1-z^\top \Sigma^{-1}z\right)^{\beta} \cdot \ind\left\{ z \in \text{int}(E)\right\},
$$
where $\Sigma:= \diag(1,\ldots,1,a_{e+1}^2,\ldots,a_d^2) \in \R^{d \times d}$ and
$$
c_1 := \frac{\Gamma\left( \frac{d}{2}+\beta + 1\right)}{\Gamma\left( \beta +1\right)\pi^{\frac{d}{2}}\prod_{i=e+1}^{d}a_i},
$$
see \autoref{ex_pearson_type_II} and recall $a_1 = \ldots = a_e = 1$. Notice that we could use $E$ itself instead of $\text{int}(E)$ as support of $f$ for $\beta \ge 0$. But, since $\partial E$ has no influence at all on the limiting behavior of $M_n$ in our setting, the consideration of $\text{int}(E)$ instead of $E$ means no loss of generality.
In this setting, we cannot state an exact limit theorem for
$$
M_n = \max_{1\le i,j \le n}|  Z_i -  Z_j|.
$$
However, by considering the projections $\overline Z_1, \overline Z_2,\ldots $ of $Z_1,Z_2,\ldots$ onto the first $e$ components and investigating
$$
\overline M_n := \max_{1\le i,j \le n}| \overline Z_i - \overline Z_j|,
$$
we can establish bounds for the unknown limit distribution, if it exists. To this end, we consider $\R^d$ as $\R^e \times \R^{d-e}$ and write $\overline z := (z_1,\ldots,z_e)$ for $z =(z_1,\ldots,z_d) \in \R^d$.
In the same way, we put $\overline Z_n := (Z_{n,1},\ldots,Z_{n,e})$ for $Z_n = (Z_{n,1},\ldots,Z_{n,d})$ and $n\in\N$. Obviously, the random variables
$\overline Z_1, \overline Z_2, \ldots $ are independent and identically distributed.
Taking some orthogonal matrix $Q_{e} \in \R^{e \times e}$ and putting $Q := \diag(Q_{e},\mathrm{I}_{d-e})$, the special form of $\Sigma$ yields
$$
f(Qz) = c_1 \cdot \left(1-z^\top Q^\top \Sigma^{-1}Qz\right)^{\beta} = c_1 \cdot \left(1-z^\top \Sigma^{-1}z\right)^{\beta} = f(z)
$$
for each $z \in \text{int}(E)$, and we can conclude that the distribution of $\overline{Z}_1,\overline{Z}_2,\ldots $
is spherically symmetric on the unit ball $\B^{e}$. In addition to that, the proof of \autoref{lem_several_major_axes_necessary_asymp_holds} (given in \cite{Schrempp2017}) reveals that this distribution solely depends on $d,e$ and $\beta$, \emph{not} on $a_{e+1},\ldots,a_d$. \\

The great advantage of assuming $a_1 = \ldots =a_e = 1$ is that we can directly apply Corollary 3.7 in~\cite{Lao2010} for the maximum distance of the random points $\overline Z_1,\overline Z_2,\ldots$  lying in the $e$-dimensional unit ball $\B^e$.
For this purpose we write $\omega_e$ for its volume and obtain the following result:
\begin{lemma}
\label{lem_several_major_axes_necessary_asymp_holds}
With
\begin{align*}
  a   %&:= \frac{\Gamma\left( \frac{d}{2}+\beta+1\right)\pi^{-\frac{e}{2}}}{\Gamma\left( \frac{d-e}{2}+\beta+1\right)}\cdot \frac{e \cdot \omega_{e} \cdot 2^{\beta+\frac{d-e}{2}}}{\beta+\frac{d-e}{2} + 1}
  &:= \frac{\Gamma\left( \frac{d}{2}+\beta+1\right)}{\Gamma\left( \frac{d-e}{2}+\beta+2\right)}\cdot \pi^{-\frac{e}{2}} \cdot e \cdot \omega_{e} \cdot 2^{\frac{d-e}{2}+\beta}
  \qquad \qquad \text{and} \qquad \qquad
  \alpha := \frac{d-e}{2} + \beta + 1,
\end{align*}
we have
$$
\P\big( 1 - |\overline Z_1| \le s\big) \sim as^\alpha
$$
as $s \downarrow 0$
\end{lemma}
The proof of this lemma can be found in Subsection 6.2.2 of \cite{Schrempp2017}.
In view of  Corollary 3.7 in~\cite{Lao2010} we define
$$
\sigma := \frac{2^{e-2}\Gamma\left( \frac{e}{2}\right)a^2\Gamma(\alpha+1)^2}{\sqrt{\pi}\Gamma\left( \frac{e+1}{2}+2\alpha\right)}
$$
with $a$ and $\alpha$ given by \autoref{lem_several_major_axes_necessary_asymp_holds}, and we put
\begin{align}
  b_n &:= \left( \frac{\sigma}{2} \right)^{\frac{2}{2d-e+4\beta+3}}\cdot n^{\frac{4}{2d-e+4\beta+3}}, \quad n \ge 1. \label{eq_several_maor_axes_ell_def_b_n}
  \intertext{Furthermore, we let}
  G(t) &:= 1 - \exp\left( -t^{\frac{2d  - e + 4 \beta +3}{2}}\right)
  \label{eq_several_maor_axes_ell_def_G}
\end{align}
for $t \ge 0$.
With Corollary 3.7 in~\cite{Lao2010} and \autoref{lem_several_major_axes_necessary_asymp_holds} we get
\begin{equation}
\label{eq_several_major_axes_conv_projections}
\P\big(b_n(2 - \overline M_n ) \le t\big)\to G(t).
\end{equation}

But, since our focus lies on the asymptotic behavior of of $\max_{1\le i,j \le n}|  Z_i -  Z_j|$, \emph{not} on that of $\max_{1\le i,j \le n}| \overline Z_i - \overline Z_j|$, we have to find some useful relation between these two random variables. The key to success will be the following lemma, which provides bounds for $|x-y|$, $x,y \in E$, that depend merely on $\overline x$, $\overline y$ and the half-axis $a_{e+1}$.
\begin{lemma}
\label{lem_no_unique_major_axis_geom_ineq}
Putting
$$
g(\overline x, \overline y) := \sqrt{\big(|\overline x| + |\overline y|\big)^2 + 2a_{e+1}^2\big( 2- |\overline x|^2 - |\overline y|^2\big)},
$$
we have
$$
|\overline x - \overline y| \le |x-y| \le g(\overline x, \overline y)
$$
for all $x,y \in E$.
\end{lemma}
The proof of this lemma can be found in Subsection 6.2.2 of \cite{Schrempp2017}.
Using the convergence given in~\eqref{eq_several_major_axes_conv_projections} and \autoref{lem_no_unique_major_axis_geom_ineq}, we can now state the main result of this section:
\vspace{2mm}
\begin{theorem}
\label{thm_several_major_axes_main_thm}
Under the standing assumptions of this section we have
\begin{align}
  G(t) \ &\le \ \liminf_{n \to \infty} \P\big( b_n( 2- M_n) \le t\big)
  \ \le  \ \limsup_{n \to \infty} \P\big( b_n( 2- M_n) \le t\big)\label{eq_several_major_axes_assertion_main_theorem}
  \ \le \ G\left( \frac{t}{ 1-a_{e+1}^2}\right), \qquad t \ge 0,
\end{align}
where $b_n$ and $G$ are given in \eqref{eq_several_maor_axes_ell_def_b_n} and \eqref{eq_several_maor_axes_ell_def_G}, respectively.
\end{theorem}
\vspace{2mm}
See \cite{Schrempp2017} for the results of a simulation study.
Before we give the proof of \autoref{thm_several_major_axes_main_thm}, we want to state an important corollary:
\begin{corollary}
From \autoref{thm_several_major_axes_main_thm} we immediately know that the sequence
$$
\left( n^{\frac{4}{2d-e+4\beta+3}}(2-M_n) \right)_{n \in \N}
 $$
 is tight. So, if there are a positive sequence $(a_n)_{n \in \N}$ and a non-degenerate distribution function $F$ with $\P \big(a_n(2-M_n)\le t\big) \to F(t)$, $t \ge 0$, we can conclude that $a_n \sim c \cdot n^{\frac{4}{2d-e+4\beta+3}}$ for some fixed $c \in \R$.
\end{corollary}

\begin{proof}[of \autoref{thm_several_major_axes_main_thm}]
From \autoref{lem_no_unique_major_axis_geom_ineq} we have
$$
|\overline Z_i - \overline Z_j| \le  |Z_i-Z_j| \le g(\overline Z_i,\overline Z_j)
$$
for all $i,j \in \N$. These inequalities imply
$$
\max_{1\le i,j \le n} |\overline Z_i - \overline Z_j|
 \le \max_{1\le i,j \le n} |Z_i-Z_j|
 \le \max_{1\le i,j \le n} g(\overline Z_i,\overline Z_j)
$$
and thus
\begin{equation}
\label{eq_no_unique_major_axis_ineq_minima}
\min_{1\le i,j \le n}\left\{ 2 - g(\overline Z_i ,\overline Z_j) \right\}
 \le \min_{1\le i,j \le n}\left\{ 2 - | Z_i - Z_j| \right\}
 \le \min_{1\le i,j \le n}\left\{ 2 - | \overline  Z_i - \overline  Z_j| \right\}.
\end{equation}
Using~\eqref{eq_several_major_axes_conv_projections} and the upper inequality figuring in~\eqref{eq_no_unique_major_axis_ineq_minima} yields
\begin{align*}
  \P\left( b_n\Big( 2- \max_{1\le i,j \le n} |Z_i-Z_j|\Big) \le t\right)
  & = \P\left( 2- \max_{1\le i,j \le n} |Z_i-Z_j| \le \frac{t}{b_n}\right)\\
  & = \P\left( \min_{1\le i,j \le n}\left\{ 2 - | Z_i - Z_j| \right\} \le \frac{t}{b_n}\right)\\
  & \ge \P\left( \min_{1\le i,j \le n}\left\{ 2 - | \overline Z_i - \overline Z_j| \right\} \le \frac{t}{b_n}\right)\\
  & = \P\left(b_n\Big(2 - \max_{1\le i,j \le n}| \overline Z_i - \overline Z_j| \Big) \le t\right)\\
  & \to G(t).
\end{align*}
Hence, the lower bound stated in~\eqref{eq_several_major_axes_assertion_main_theorem} has already been obtained.
To establish the upper bound in~\eqref{eq_several_major_axes_assertion_main_theorem}, we consider $\R^e \times \R^{e}$. For $(\overline x, \overline y)\in \B^e \times \B^{e} $ close to $ \mathbf{a} :=(-1,\mathbf{0} ,1,\mathbf{0} ) \in \R^{2e}$ we have, putting
$$
c := 1 - a_{e+1}^2,
$$
the multivariate Taylor series expansions
\begin{align*}
  2 - g(\overline x, \overline y) &= c \cdot\left( 2 + x_1 - y_1\right) + o\big(|( \overline x, \overline y) - \mathbf{a} |\big),\\
  2 - |\overline x - \overline y| &= \left( 2 + x_1 - y_1 \right) + o\big(|( \overline x, \overline y) - \mathbf{a} |\big),
  \intertext{and}
  \frac{2 - g(\overline x,\overline y)}{2 - | \overline x - \overline y|} &= c + o\big(|( \overline x, \overline y) - \mathbf{a} |\big).
\end{align*}
By symmetry, we can conclude that
$$
\frac{2 - g(\overline x,\overline y)}{2 - | \overline x - \overline y|} \to c
$$
for $(\overline x, \overline y) \in \B^e \times \B^{e}$ with
$(\overline x, \overline y) \to (\mathbf{a}^*,-\mathbf{a}^*  )$ and $\mathbf{a}^* \in \partial\B^{e} $. Furthermore, the symmetry guarantees that, for each $\delta \in (0,c)$, we can find a positive $\varepsilon$ so that
$$
c - \delta \le \frac{2 - g(\overline x,\overline y)}{2 - | \overline x - \overline y|}
$$
for all $(\overline x,\overline y) \in \B^{e} \times \B^{e}$ with $|\overline x - \overline y| \ge 2-\varepsilon$. For $n \in \N$, we write $\overline Z_{n}^{1}$ and $\overline Z_{n}^{2}$ for those elements of $\left\{ \overline Z_1,\ldots,\overline Z_n\right\}$ with
$$
\max_{1 \le i,j\le n}|\overline Z_i - \overline Z_j| = \big|\overline Z_{n}^{1} - \overline Z_{n}^{2}\big|.
$$
Based on these two random variables, we define for $\varepsilon$ given above the set
$$
A_{n,\varepsilon} := \left\{ \big|\overline Z_{n}^{1} - \overline Z_{n}^{2}\big| > 2 - \varepsilon \right\}.
$$
Obviously, $ \P(A_{n,\varepsilon}^c) \to 0$, and the event $A_{n,\varepsilon}$ entails
$$
c - \delta \le \frac{2 - g\big(\overline Z_{n}^{1},\overline Z_{n}^{2}\big)}{2 - \big| \overline Z_{n}^{1} - \overline Z_{n}^{2}\big|}.
$$
Together with the lower inequality given in~\eqref{eq_no_unique_major_axis_ineq_minima} we obtain
\begin{align*}
  \P\left( b_n\Big( 2- \max_{1\le i,j \le n} |Z_i-Z_j|\Big) \le t\right)
  %=\; &\P\left( 2- \max_{1\le i,j \le n} |Z_i-Z_j| \le  t , A_{n,\varepsilon}\right) + \P\left(b_n\left( 2- \max_{1\le i,j \le n} |Z_i-Z_j|\right) \le t , A_{n,\varepsilon}^c\right)\\
  \le\ &\P\left( b_n\min_{1\le i,j \le n}\left\{ 2 - | Z_i - Z_j| \right\} \le  t , A_{n,\varepsilon}\right) + \P(A_{n,\varepsilon}^c) \\
  \le\ & \P\left( b_n\min_{1\le i,j \le n}\left\{ 2 - g(\overline Z_i,\overline Z_j) \right\} \le t , A_{n,\varepsilon}\right) + \P(A_{n,\varepsilon}^c)\\
  =\ &\P\left( b_n\min_{1\le i,j \le n}\left\{ \left( 2 - | \overline Z_i - \overline Z_j|\right) \cdot \frac{2 - g(\overline Z_i,\overline Z_j)}{2 - | \overline Z_i - \overline Z_j|} \right\} \le t , A_{n,\varepsilon}\right) + \P(A_{n,\varepsilon}^c) \\
  \le \ &\P\left( b_n\min_{1\le i,j \le n}\left\{ \left( 2 - | \overline Z_i - \overline Z_j|\right) \cdot (c-\delta)\right\} \le  t , A_{n,\varepsilon}\right) + \P(A_{n,\varepsilon}^c) \\
  \le \ &\P\left(b_n\Big(2 - \max_{1\le i,j \le n}| \overline Z_i - \overline Z_j| \Big) \le \frac{t}{c-\delta}\right) + \P(A_{n,\varepsilon}^c) \\
  \to \ & G\left( \frac{t}{c-\delta}\right).
\end{align*}
Since $\delta$ can be chosen arbitrarily close to $0$, the continuity of $G$ implies
$$
\limsup_{n\to \infty} \P\left( b_n\Big( 2- \max_{1\le i,j \le n} |Z_i-Z_j|\Big) \le t\right) \le G\left( \frac{t}{c} \right),
$$
and the proof is finished.
\qed
\end{proof}

\section{Appendix}
\label{sec_appendix}
\begin{proof}[of \autoref{lem_first_part_deriv_are_0_and_Hessian_pos_def}]
We only consider $i= \ell$. It is clear that $H_\ell$ is symmetric, since $s^\ell$ is a twice continuously differentiable function.
From \autoref{cond_unique_diameter} we know that
\begin{equation}
\label{eq_E_subset_ball_Radius_2a_at_right_pole}
E \subset B_{2a}\big( (a,\mathbf{0} )\big) \qquad \text{and} \qquad  E \cap \partial B_{2a}\big( (a,\mathbf{0} )\big) = \left\{ (-a,\mathbf{0} )\right\}.
\end{equation}
Writing $O_t := \left\{ \widetilde z \in \R^{d-1}: |\widetilde z|<2a \right\}$ and defining the mapping $t: O_t \to \R, \widetilde z \mapsto a - \sqrt{4a^2 - z_2^2- \ldots -z_d^2}$, the boundary of $B_{2a}\big( (a,\mathbf{0} )\big) $ in $ \left\{ z_1 < a\right\}$ can be parameterized as a hypersurface via
$$
\mathbf{t}:
\begin{cases}
  O_t \to \R^d,\\
  \widetilde z \mapsto \big(\ t( \widetilde z)\ ,\ \widetilde z\ \big).
\end{cases}
$$
For $j,k \in \left\{ 2,\ldots,d\right\}$, we obtain
\begin{align*}
  t_j(\widetilde z) &= (4a^2 - z_2^2 - \ldots - z_d^2)^{-\frac{1}{2}}\cdot z_j,\\
  t_{jk}(\widetilde z) &= (4a^2 - z_2^2 - \ldots - z_d^2)^{-\frac{3}{2}}\cdot z_jz_k + (4a^2 - z_2^2 - \ldots - z_d^2)^{-\frac{1}{2}}\cdot \delta_{jk}.
\end{align*}
Hence, $\nabla t(\mathbf{0} ) = \mathbf{0} $, and the Hessian of $t$ at $\mathbf{0} $ is given by $H_t := \frac{1}{2a}\mathrm{I}_{d-1} $. So, the second-order Taylor series expansion of $t$ at this point has the form
\begin{equation}
\label{eq_Taylor_function_t}
t(\widetilde z) = -a + \mathbf{0}^\top\widetilde z + \frac{1}{2}\widetilde z^\top H_t \widetilde z + R_t(\widetilde z),
\end{equation}
where $R_t (\widetilde z) = o\big(|\widetilde z|^2 \big)$. Furthermore, we have
\begin{equation}
\label{eq_Taylor_s_l}
  s^{\ell}(\widetilde z) =  -a +  \nabla s^{\ell}(\mathbf{0} )^\top \widetilde z +  \frac{1}{2}\widetilde z^\top H_\ell \widetilde z + R_\ell(\widetilde z),
\end{equation}
where $R_\ell (\widetilde z) = o\big(|\widetilde z|^2 \big)$.
In view of \eqref{eq_E_subset_ball_Radius_2a_at_right_pole} and \autoref{cond_shape_pole_caps}, we have $t(\widetilde z) <-a + s^{\ell}(\widetilde z)$ for each $\widetilde z \in O_\ell \backslash\left\{ \mathbf{0} \right\}$ (observe that \eqref{eq_E_subset_ball_Radius_2a_at_right_pole} ensures $O_\ell \subset O_t$). Using \eqref{eq_Taylor_function_t} and\eqref{eq_Taylor_s_l}, this inequality can be rewritten as
$$
-a +  \frac{1}{2}\widetilde z^\top H_t \widetilde z + R_t(\widetilde z)< -a +  \nabla s^{\ell}(\mathbf{0} )^\top \widetilde z +  \frac{1}{2}\widetilde z^\top H_\ell \widetilde z + R_\ell(\widetilde z),
$$
and hence
$$
 0< \nabla s^{\ell}(\mathbf{0} )^\top \widetilde z +  \frac{1}{2}\widetilde z^\top (H_\ell - H_t) \widetilde z + \big(R_\ell(\widetilde z) -R_t(\widetilde z) \big)
$$
for each $\widetilde z \in O_\ell\backslash\left\{ \mathbf{0} \right\}$. Since $R_\ell (\widetilde z)-R_t (\widetilde z) = o\big(|\widetilde z|^2 \big)$, this inequality shows $\nabla s^{\ell}(\mathbf{0} ) = \mathbf{0} $ and that the matrix $H_\ell - H_t$ is positive definite. Remembering $H_t = \frac{1}{2a}\mathrm{I}_{d-1} $, $H_\ell$ has to be positive
definite, too, and all eigenvalues of $H_\ell$ have to be larger than $1/2a$.
\qed
\end{proof}

Now we will show that \autoref{cond_A_eta_pos_semi_definite} really ensures the unique diameter of $E$ `close to the poles':

\begin{lemma}
\label{lem_suff_con_unique_diam}
Under Conditions~\ref{cond_shape_pole_caps} and \ref{cond_A_eta_pos_semi_definite}, \eqref{eq_cond_unique_diameter} holds true for $E$ replaced with $E \cap \left\{ |z_1| > a - \delta\right\}$ and $\delta >0 $ sufficiently small.
\end{lemma}
\begin{proof}
Since the diameter of $E$ cannot be determined by interior points, it suffices to investigate points on the boundaries $M_\ell$ and $M_r$ of the pole-caps of $E$.
To this end, let $(\widetilde x ,\widetilde y)\in O_\ell \times O_r\backslash \left\{ \mathbf{0}\right\}$. Invoking \eqref{eq_representation_E_l_Taylor} and \eqref{eq_representation_E_r_Taylor} and putting
 $$
 \Xi:= \frac{1}{2}\left(\widetilde x^\top  H_\ell \widetilde x + \widetilde y^\top  H_r \widetilde y \right) + R_\ell (\widetilde x) + R_r (\widetilde y),
 $$
we get
\begin{align*}
 \big| (-a + s^\ell(\widetilde x),\widetilde x ) - \big( a -s^r(\widetilde y),\widetilde y\big) \big|^2\
 = \ &\Big| \Big(-a + \frac{1}{2}\widetilde x^\top  H_\ell \widetilde x + R_\ell (\widetilde x),\widetilde x \Big) - \Big( a - \frac{1}{2}\widetilde y^\top  H_r \widetilde y - R_r (\widetilde y),\widetilde y\Big) \Big|^2\\
 = \ &\left( -2a + \Xi\right)^2 + |\widetilde x- \widetilde y|^2\\
 = \ &4a^2 - 4a\Xi + \Xi^2 + |\widetilde x- \widetilde y|^2\\
 = \ &4a^2 - \left( 4a\Xi -\Xi^2 \right) +|\widetilde x|^2 + | \widetilde y|^2 - 2\widetilde x^\top  \widetilde y.
 \end{align*}
\autoref{lem_inequality_for_Xi} will show that
\begin{equation}
\label{eq_inequality_for_Xi}
 4a\Xi - \Xi^2 > 2a\con\left( \widetilde x^\top  H_\ell \widetilde x + \widetilde y^\top  H_r \widetilde y \right)
\end{equation}
for every $(\widetilde x,\widetilde y) \ne \mathbf{0}$ sufficiently close to $\mathbf{0} $.
Representing the points $\widetilde x$ and $\widetilde y$ in terms of the bases $\left\{ \mathbf{u}_2^\ell,\ldots,\mathbf{u}_d^\ell \right\}$ and $\left\{ \mathbf{u}_2^r,\ldots,\mathbf{u}_d^r \right\}$, namely $\widetilde x = U_\ell \alpha$ and $\widetilde y = U_r \beta$, \eqref{eq_cond_1_unique_diam} gives
\begin{align*}
  \big| (-a + s^\ell(\widetilde x),\widetilde x ) - \big( a -s^r(\widetilde y),\widetilde y\big) \big|^2\
 < \ &4a^2 - 2a\con\left( \widetilde x^\top  H_\ell \widetilde x + \widetilde y^\top  H_r \widetilde y \right) + |\widetilde x|^2 + | \widetilde y|^2 - 2\widetilde x^\top  \widetilde y \\
 = \ &4a^2 - 2a\con\left( \alpha^\top  U_\ell^\top  H_\ell U_\ell \alpha  + \beta^\top  U_r^\top  H_r U_r \beta \right) + |U_\ell\alpha|^2 + | U_r \beta|^2 - 2\alpha^\top  U_\ell^\top  U_r \beta\\
 = \ &4a^2 - 2a\con\left( \alpha^\top  D_\ell \alpha  + \beta^\top  D_r \beta \right) - 2\alpha^\top  U_\ell^\top  U_r \beta+ |\alpha|^2 + |\beta|^2\\
 \le \ &4a^2.
 \end{align*}
Thus, for $\delta>0$ sufficiently small, the only pair of points in $E \cap \left\{ |z_1| > a - \delta\right\}$ with distance $2a$ is given by $(-a,\mathbf{0} )$ and $(a,\mathbf{0} )$, and the proof is finished.
\qed
\end{proof}

It remains to prove the validity of \eqref{eq_inequality_for_Xi}.
\begin{lemma}
\label{lem_inequality_for_Xi}
For  $\con \in \left( 0,1\right)$ and $(\widetilde x,\widetilde y) \ne \mathbf{0}$ sufficiently close to  $\mathbf{0} $ we have
\begin{align*}
 4a\Xi - \Xi^2 > 2a\con\left( \widetilde x^\top  H_\ell \widetilde x + \widetilde y^\top  H_r \widetilde y \right).% \tag{\ref{eq_inequality_for_Xi}}
\end{align*}
\end{lemma}
\begin{proof}
Let $\varepsilon := \frac{1-\con}{2}>0$. Without loss of generality we assume  $\widetilde x \ne \mathbf{0} $.
For $\widetilde x$ sufficiently close to $\mathbf{0} $, \eqref{eq_Courant_Fischer} and $R_\ell (\widetilde x) = o\big(|\widetilde x|^2 \big)$ lead to
 $$
\big| R_\ell (\widetilde x)\big| <  \frac{\varepsilon}{2}\kappa_2^\ell|\widetilde x|^2 \le \frac{\varepsilon}{2}\widetilde x^\top  H_\ell \widetilde x,
 $$
 whence
 $$
 \frac{1}{2}\widetilde x^\top  H_\ell \widetilde x + R_\ell (\widetilde x) > \frac{1}{2}\widetilde x^\top  H_\ell \widetilde x - \frac{\varepsilon}{2}\widetilde x^\top  H_\ell \widetilde x = \frac{1-\varepsilon}{2}\widetilde x^\top  H_\ell \widetilde x .
 $$
 By the same reasoning for $\widetilde y$ we get
 $$
 \frac{1}{2}\widetilde y^\top  H_r \widetilde y + R_r (\widetilde y) \ge  \frac{1-\varepsilon}{2}\widetilde y^\top  H_r \widetilde y.
 $$
 Observe that, in the line above, equality holds if $\widetilde y = \mathbf{0} $.
 Putting both inequalities together yields
 $$
 \Xi > \frac{1-\varepsilon}{2}\left(\widetilde x^\top  H_\ell \widetilde x + \widetilde y^\top  H_r \widetilde y \right)
 $$
 and thus
 \begin{equation}
 \label{eq_inequality_Xi_1}
   4a\Xi > 2a(1-\varepsilon)\left(\widetilde x^\top  H_\ell \widetilde x + \widetilde y^\top  H_r \widetilde y \right).
 \end{equation}
 Since close to $\mathbf{0}$  both $\big|R_\ell (\widetilde x)\big| \le \frac{\kappa_d^\ell}{2} |\widetilde x|^2$ and $\big|R_r (\widetilde y) \big|\le \frac{\kappa_d^r }{2}|\widetilde y|^2$ hold true, \eqref{eq_Courant_Fischer} gives
 \begin{align*}
   \Xi^2 &= \left( \frac{1}{2}\left(\widetilde x^\top  H_\ell \widetilde x + \widetilde y^\top  H_r \widetilde y \right) + R_\ell (\widetilde x) + R_r (\widetilde y)\right)^2\\
   &\le \left( \frac{1}{2}\big( \kappa_d^\ell|\widetilde x|^2 + \kappa_d^r|\widetilde y|^2\big) + \frac{\kappa_d^\ell }{2}|\widetilde x|^2 + \frac{\kappa_d^r }{2}|\widetilde y|^2 \right)^2 \\
   &= \left(  \kappa_d^\ell|\widetilde x|^2 + \kappa_d^r|\widetilde y|^2 \right)^2 \\
   &\le  \max\left\{ \kappa_d^\ell,\kappa_d^r\right\}^2 \left(  |\widetilde x|^2 +|\widetilde y|^2 \right)^2.
 \end{align*}
 Using \eqref{eq_Courant_Fischer} again yields
 $$
  0 \le \frac{\left(  |\widetilde x|^2 +|\widetilde y|^2 \right)^2}{\widetilde x^\top  H_\ell \widetilde x + \widetilde y^\top  H_r \widetilde y } \le
  \frac{\left(  |\widetilde x|^2 +|\widetilde y|^2 \right)^2}{\kappa_2^\ell  |\widetilde x|^2  + \kappa_2^r  |\widetilde y|^2}.
 $$
 Since the fraction on the right-hand side tends to $0$ as $(\widetilde x,\widetilde y) \to \mathbf{0}$ we infer
 \begin{equation}
 \label{eq_inequality_Xi_2}
   \Xi^2 \le 2 a \varepsilon \left(\widetilde x^\top  H_\ell \widetilde x + \widetilde y^\top  H_r \widetilde y \right)
 \end{equation}
 for all $(\widetilde x,\widetilde y)$ sufficiently close to $\mathbf{0} .$
 From \eqref{eq_inequality_Xi_1} and \eqref{eq_inequality_Xi_2} we deduce that
 \begin{align*}
  4a\Xi - \Xi^2 &> 2a(1-\varepsilon)\left(\widetilde x^\top  H_\ell \widetilde x + \widetilde y^\top  H_r \widetilde y \right) - 2 a \varepsilon \left(\widetilde x^\top  H_\ell \widetilde x + \widetilde y^\top  H_r \widetilde y \right) \\
  &= 2a(1-2\varepsilon)\left(\widetilde x^\top  H_\ell \widetilde x + \widetilde y^\top  H_r \widetilde y \right),
 \end{align*}
 and since $1-2\varepsilon= 1- 2\frac{1-\con}{2}= \con$, the proof is finished.
 \qed
\end{proof}

Now we want to show that the matrix $A(1)$ is necessarily positive semi-definite. Otherwise, we would obtain a contradiction to \autoref{cond_unique_diameter}.

\begin{lemma}
\label{lem_A_1_has_to_be_pos_def}
Under Conditions~\ref{cond_unique_diameter} and \ref{cond_shape_pole_caps} we have $A(1) \ge 0$.
\end{lemma}
\begin{proof}
Assuming $A(1) \ngeq 0$, there exists $z \in \R^{2(d-1)}$ with
$ z^\top A(1)z <0$. Then, we can also find an $\con^*>1$ with
\begin{align*}
  z^\top A(\con^*)z
  &= z^\top \big( A(1) + 2a(\con^{*}-1) \diag(D_\ell,D_r) \big)z\\
  &= z^\top A(1)z + (\con^{*}-1) 2a z^\top \diag(D_\ell,D_r)z \\
  &< 0,
\end{align*}
which entails $A(\con^*) \ngeq 0$. Notice that $(\con^{*}-1)2az^\top \diag(D_\ell,D_r)z > 0$ can be made arbitrarily small by choosing $\con^{*}$  sufficiently close to $1$.
In a similar way as in the proof of \autoref{lem_inequality_for_Xi}, one can show
$$
 4a\Xi - \Xi^2 < 2a\con^*\left( \widetilde x^\top  H_\ell \widetilde x + \widetilde y^\top  H_r \widetilde y \right)
$$
for all $(\widetilde x,\widetilde y) \ne \mathbf{0}$ sufficiently close to $\mathbf{0} $.
As in the proof of \autoref{lem_suff_con_unique_diam} we obtain
 \begin{align}
  &\big| (-a + s^\ell(\widetilde x),\widetilde x ) - \big( a -s^r(\widetilde y),\widetilde y\big) \big|^2
 > \ 4a^2 - 2a\con^*\left( \widetilde x^\top  H_\ell \widetilde x + \widetilde y^\top  H_r \widetilde y \right) + |\widetilde x|^2 + | \widetilde y|^2 - 2\widetilde x^\top  \widetilde y.
 \label{eq_proof_A_1_pos_definit_help_1}
 \end{align}
Because of $A(\con^*) \ngeq 0$ we can find $\alpha,\beta \in \R^{d-1}$ arbitrarily close to $\mathbf{0} $ with
\begin{align*}
\Big(  \alpha^\top  \ , \ \beta^\top \Big)
A(\con^*) \begin{pmatrix}
  \alpha \\
  \beta
\end{pmatrix} &< 0.
\end{align*}
This inequality can be rewritten to
\begin{equation}
\label{eq_proof_A_1_pos_definit_help_2}
-2a\con^*\left( \alpha^\top D_\ell\alpha + \beta^\top D_r \beta\right)-2\alpha^\top U_\ell^\top U_r\beta+|\alpha|^2 + |\beta|^2 > 0.
\end{equation}
If we choose $|\alpha|$ and $|\beta|$ small enough, we have $\widetilde x := U_\ell\alpha \in O_\ell$ and $\widetilde y := U_r \beta \in O_r$.
Putting \eqref{eq_proof_A_1_pos_definit_help_1} and \eqref{eq_proof_A_1_pos_definit_help_2} together yields
 \begin{align*}
  \big| (-a + s^\ell(\widetilde x),\widetilde x ) - \big( a -s^r(\widetilde y),\widetilde y\big) \big|^2\
 > \ &4a^2 - 2a\con^*\left( \widetilde x^\top  H_\ell \widetilde x + \widetilde y^\top  H_r \widetilde y \right) + |\widetilde x|^2 + | \widetilde y|^2 - 2\widetilde x^\top  \widetilde y \\
 = \ &4a^2 - 2a\con^*\left( \alpha^\top  U_\ell^\top  H_\ell U_\ell \alpha  + \beta^\top  U_r^\top  H_r U_r \beta \right)- 2\alpha^\top  U_\ell^\top  U_r \beta + |U_\ell\alpha|^2 + | U_r \beta|^2\\
 = \ &4a^2 - 2a\con^*\left( \alpha^\top  D_\ell \alpha  + \beta^\top  D_r \beta \right)- 2\alpha^\top  U_\ell^\top  U_r \beta + |\alpha|^2 + |\beta|^2\\
 > \ &4a^2.
 \end{align*}
 This inequality contradicts \autoref{cond_unique_diameter}, and the proof is finished.
 \qed
\end{proof}

The following lemma shows that inequality~\eqref{eq_suff_cond_princ_curv} is sufficient for \autoref{cond_A_eta_pos_semi_definite}:

\begin{lemma}
\label{lem_suff_cond_princ_curv}
If \eqref{eq_suff_cond_princ_curv} holds true,
then \autoref{cond_A_eta_pos_semi_definite} is fulfilled.
\end{lemma}
\begin{proof}
Inequality~\eqref{eq_suff_cond_princ_curv} ensures the existence of an $\con^{*} \in (0,1)$ with
\begin{equation}
\label{eq_suff_cond_princ_curv_help}
 \frac{1}{\kappa_2^\ell} + \frac{1}{\kappa_2^r} = 2a\con^{*}.
\end{equation}
Applying \eqref{eq_Courant_Fischer} (with $H_\ell$ and $H_r$ replaced with the matrices $D_\ell$ and $D_r$, respectively), using \eqref{eq_suff_cond_princ_curv_help} and some obvious transformations yield
\begin{align*}
  &2a\con^{*}\left( \alpha^\top D_\ell\alpha + \beta^\top D_r \beta\right)+2\alpha^\top U_\ell^\top U_r\beta-|\alpha|^2 - |\beta|^2
  \ge \ \left| \sqrt{\frac{\kappa_2^\ell}{\kappa_2^r}}U_\ell\alpha + \sqrt{\frac{\kappa_2^r}{\kappa_2^\ell}}U_r\beta \right|^2
  \ge\  0.
\end{align*}
Consequently, \autoref{cond_A_eta_pos_semi_definite} holds with $\con = \con^{*}$, see \eqref{eq_cond_1_unique_diam}.
\qed
\end{proof}

As mentioned before, \eqref{eq_suff_cond_princ_curv} is only sufficient for the unique diameter close to the poles, \emph{not} necessary. See Example 3.13 in \cite{Schrempp2017} for an illustration of a set with unique diameter between $(-a,\mathbf{0} )$ and $(a,\mathbf{0} )$ for which inequality~\eqref{eq_suff_cond_princ_curv} is \emph{not} fulfilled.

\begin{acknowledgements}
This paper is based on the author's doctoral dissertation written under the guidance of Prof. Dr. Norbert Henze.
The author wishes to thank Norbert Henze for bringing this problem to his attention and for helpful
discussions.
\end{acknowledgements}

\bibliographystyle{spbasic}      % basic style, author-year citations
\bibliography{Literatur}

\end{document}